\documentclass[hidelinks,onefignum,onetabnum]{siamart251216}

\usepackage{lipsum}
\usepackage{amsfonts}
\usepackage{graphicx}
\usepackage{epstopdf}
\usepackage{algorithmic}
\ifpdf{}
  \DeclareGraphicsExtensions{.eps,.pdf,.png,.jpg}
\else
  \DeclareGraphicsExtensions{.eps}
\fi

\newsiamremark{remark}{Remark}
\newsiamremark{hypothesis}{Hypothesis}
\crefname{hypothesis}{Hypothesis}{Hypotheses}
\newsiamthm{claim}{Claim}
\newsiamremark{fact}{Fact}
\crefname{fact}{Fact}{Facts}

\headers{Energy-minimizing domain decomposition}{Muhammad Hassan, Benjamin Stamm, and Lambert Theisen}

\title{Energy-minimizing domain decomposition\thanks{Submitted to the editors DATE.%
}}

\author{Muhammad Hassan\thanks{Technische Universität München, Department of Mathematics, Boltzmannstrasse 3, Garching 85748, Germany
  (\email{muhammad.hassan@cit.tum.de}).}
\and Benjamin Stamm\thanks{Department of Mathematics, University of Stuttgart, Germany 
  (\email{best@ians.uni-stuttgart.de}).}
\and Lambert Theisen\thanks{Applied and Computational Mathematics, RWTH Aachen University, Germany
  (\url{https://thsn.dev}, \email{lambert.theisen@rwth-aachen.de}).}
}

\usepackage{amsopn}

\usepackage[svgnames]{xcolor}
\usepackage[english]{babel}
\usepackage[babel=true,expansion=alltext,protrusion=alltext-nott,final]{microtype}
\usepackage[utf8]{inputenc}
\usepackage{textcomp}

\usepackage{mathrsfs}  
\usepackage{amssymb}

\usepackage[varqu,varl]{inconsolata}
\usepackage{mathtools}
\usepackage[T1]{fontenc}
\usepackage{textcomp}
\usepackage{booktabs}
\usepackage[inline]{enumitem}
\numberwithin{equation}{section}
\usepackage[normalem]{ulem}

\usepackage{float}
\usepackage{silence}
\usepackage{caption}
\usepackage{subcaption}
\usepackage{xspace}
\usepackage[scientific-notation=true]{siunitx}
\usepackage{pstricks}
\usepackage{hyperref}
\usepackage{tikz}
\usetikzlibrary{positioning}
\usetikzlibrary{calc}
\usepackage{pgfplots}
\pgfplotsset{width=10cm,compat=1.9}
\usepackage{pgfplotstable}
\usepackage{bm}
\usepackage{sidecap}
\usepackage[most]{tcolorbox}
\usepackage{lipsum}

\DeclareFontFamily{U}{mathx}{}
\DeclareFontShape{U}{mathx}{m}{n}{<-> mathx10}{}
\DeclareSymbolFont{mathx}{U}{mathx}{m}{n}
\DeclareMathAccent{\widehat}{0}{mathx}{"70}
\DeclareMathAccent{\widecheck}{0}{mathx}{"71}

\newtheorem{convention}{Convention}

\definecolor{light_gray}{gray}{0.75}
\definecolor{lighter_gray}{gray}{0.5}
\colorlet{light_blue}{blue!20}
\definecolor{dark_green}{rgb}{0.0, 0.6, 0.0}
\definecolor{royal_blue}{rgb}{0.0, 0.22, 0.66}
\definecolor{salmon}{rgb}{1.0, 0.55, 0.41}
\definecolor{gold}{rgb}{0.8, 0.63, 0.21}
\definecolor{navy_blue}{rgb}{0.0, 0.0, 0.5}
\definecolor{crimson}{rgb}{0.79, 0.0, 0.09}
\definecolor{amethyst}{rgb}{0.6, 0.4, 0.8}
\definecolor{alizarin}{rgb}{0.82, 0.1, 0.26}
\definecolor{amaranth}{rgb}{0.9, 0.17, 0.31}
\definecolor{azure}{rgb}{0.0, 0.5, 1.0}
\definecolor{canaryyellow}{rgb}{0.82, 0.41, 0.12}
\definecolor{carrotorange}{rgb}{0.8, 0.33, 0.0}
\definecolor{cadmiumgreen}{rgb}{0.0, 0.42, 0.24}
\definecolor{copper}{rgb}{0.72, 0.45, 0.2}
\definecolor{aqua}{rgb}{0.5, 1.0, 0.83}
\definecolor{awesome}{rgb}{1.0, 0.13, 0.32}
\definecolor{candyapplered}{rgb}{1.0, 0.03, 0.0}
\definecolor{caribbeangreen}{rgb}{0.0, 0.8, 0.6}
\definecolor{indigo}{rgb}{0.0, 0.25, 0.42}

\DeclareMathOperator*{\argmin}{arg\,min}

\newcommand{\R}{\mathbb R}
\newcommand{\N}{\mathbb N}

\newcommand{\bw}{{\bf w}}

\newcommand{\cE}{{\mathcal E}}

\newcommand{\rH}{{\mathrm H}}

\renewcommand{\phi}{\varphi}

\renewcommand{\bold}[1]{{\bf #1}}

\newcommand{\en}{\cE}

\newcommand{\vspan}{\textrm{span}}

\newcommand{\Vapp}{V_{\rm approx}}

\newcommand{\Nloc}{N_{\rm loc}}
\newcommand{\Nkry}{N_{\rm kry}}
\newcommand{\Nevp}{N_{\rm evp}}
\newcommand{\Nnlin}{N_{\rm nl}}
\newcommand{\Nscf}{N_{\rm scf}}

\newcommand{\Vik}{V_i^{(k)}}
\newcommand{\Vjk}{V_j^{(k)}}
\newcommand{\Vzk}{V_0^{(k)}}

\newcommand{\fiv}{{\bf{f}}^{i}}
\newcommand{\fz}{{\bf{f}}^{0}}
\newcommand{\yi}{{\bf{y}}^{i}}
\newcommand{\bA}{{\bf{A}}}
\newcommand{\Ai}{{\bf{A}}^{\! i}}
\newcommand{\Az}{{\bf{A}}^{\! 0}}
\newcommand{\Gi}{{\bf{G}}^{i}}
\newcommand{\Gz}{{\bf{G}}^{0}}
\newcommand{\bN}{{\bf{N}}}
\newcommand{\Ni}{{\bf{N}}^{i}}
\newcommand{\Nz}{{\bf{N}}^{0}}
\newcommand{\bS}{{\bf{S}}}
\newcommand{\Si}{{\bf{S}}^{i}}
\newcommand{\Sz}{{\bf{S}}^{0}}
\newcommand{\Hi}{{\bf{H}}^{i}}
\newcommand{\Hz}{{\bf{H}}^{0}}

\newcommand{\bfb}{{\bf b}}

\newcommand{\uk}{{\bf u}^{(k)}}
\newcommand{\uz}{{\bf u}^{(0)}}

\newcommand{\resk}{{\bf r}^{(k)}}
\newcommand{\resz}{{\bf r}^{(0)}}

\newcommand{\Pu}{{\bf{P}}}

\newcommand{\comp}[1]{\mathcal O(#1)}
\newcommand{\Comp}[1]{\mathcal O\big(#1\big)}

\makeatletter
\newcommand{\longdash}[1][2em]{%
  \makebox[#1]{$\m@th\smash-\mkern-7mu\cleaders\hbox{$\mkern-2mu\smash-\mkern-2mu$}\hfill\mkern-7mu\smash-$}}
\makeatother
\newcommand{\omitskip}{\kern-\arraycolsep}

\newcommand{\remove}[1]{}

\setlist[itemize,enumerate]{leftmargin=0.65cm}

\ifpdf
\hypersetup{
  pdftitle={Energy-Minimizing Domain Decomposition},
  pdfauthor={M. Hassan, B. Stamm, and L. Theisen}
}
\fi

\begin{document}

\maketitle
\begin{abstract}
We propose a novel domain decomposition framework based on energy
minimization, applicable to a wide class of linear and nonlinear partial
differential equations. This energy-minimizing domain decomposition (EMDD)
method addresses problems whose solutions are minimizers of an energy
functional, either unconstrained or constrained to a manifold, and thus
covers source and eigenvalue problems in a single formulation. Besides introducing the general framework, we derive the algorithmic realization of the method for four classes of model problems, namely linear and nonlinear source and eigenvalue problems, and derive the computational complexity of the method. Numerical experiments demonstrate the robustness and efficiency of the method across all four model problems and show that it compares favourably against existing model-specific domain decomposition methods.
\end{abstract}



\section{Introduction}

Domain decomposition (DD) methods are a cornerstone of scientific computing,
providing a systematic way to solve large-scale partial differential equations
by splitting a global problem into smaller subproblems posed on subdomains.
Being naturally suited to parallel and distributed architectures, they are
indispensable for computationally intensive simulations in engineering, physics, and the
applied sciences. A rich theory has been developed over the past decades,
particularly for linear elliptic problems, where Schwarz-type iterations and
their convergence properties are by now well understood; the monographs
~\cite{Smith1996Domain,quarteroni1999domain,toselli2004domain} witness the maturity of the field.

Despite this maturity, the DD literature remains fragmented along problem
types. Linear source problems are typically treated by additive and
multiplicative Schwarz methods, or their optimized variants, whose convergence
theory rests on linearity and coercivity of the underlying operator. For such
problems, DD is usually employed as a preconditioner for Krylov subspace
methods rather than as a stationary iterative solver, since the Krylov method
selects its residual polynomial adaptively instead of being confined to a fixed
one.

For nonlinear problems the picture is less unified, and methods differ in the
order in which linearization and decomposition are applied, and in whether DD
serves as an iterative scheme or as a preconditioner. Applying a DD iteration
directly to the nonlinear equations yields small subdomain problems that are
often easier to solve than the global one; early
contributions, including analysis, are \cite{lions1988schwarz,dryja1997nonlinear,lui1999schwarz} but is still an active branch of research~\cite{chaouqui2022linear}. DD may
also be applied within the Newton iteration to solve the linearized problems in
parallel, leading to the Newton-Krylov-Schwarz methods
\cite{cai1994domain,cai1998parallel,cai1994newton, knoll2004jacobian}. The same idea has been combined with the Finite Element Tearing and Interconnecting--Dual-Primal (FETI-DP) for nonlinear
implicit problems in structural mechanics
\cite{klawonn2014nonlinear,klawonn2017nonlinear}. A third route uses DD to
construct a preconditioned nonlinear system: ASPIN
\cite{cai2002nonlinearly,cai2002non,hwang2007class} and its multiplicative
variant MSPIN \cite{liu2015field} solve the preconditioned equation directly,
whereas RASPEN \cite{dolean2016nonlinear} supplies a preconditioner for the
exact Newton scheme.

Eigenvalue problems introduce a further difficulty even in the linear case,
since the normalization constraint and the coupling between eigenvector and
eigenvalue lend the problem a nonlinear character. Consequently, DD is most
commonly used indirectly: an outer eigensolver such as inverse iteration,
shift-and-invert Arnoldi, or Jacobi-Davidson requires repeated linear solves, for which DD supplies the preconditioner (see
\cite{hwang2010parallel,zhao2016parallel,wang2019convergence,wang2018two} for
the combination with Jacobi-Davidson or~\cite{theisenScalableTwoLevelDomain2024} for gradient-based methods). The decomposition then acts on the linear
algebra of the eigensolver rather than on the eigenvalue problem itself.
Genuinely decomposed eigensolvers, posing local eigenvalue or minimization
problems on the subdomains, are rare. An early contribution is
due to Maliassov \cite{maliassov1998schwarz}, who constructed multiplicative
and additive Schwarz analogues for the principal eigenpair of a symmetric
elliptic operator by minimizing the Rayleigh quotient successively in
overlapping subspaces associated with the subdomains; the multiplicative
version is also considered in \cite{chan2002subspace}. The method of
\cite{kalantzis2020domain} instead adapts the discrete Schur complement to the eigenvalue setting and is thus related to the Feshbach-Schur
method~\cite{bach1998quantum,dusson2021feshbach}.

The \emph{goal} of this paper is to introduce a general framework for domain
decomposition based solely on energetic considerations. Our central
observation is that a large class of problems, linear and nonlinear, source and
eigenvalue alike, can be recast as the Euler-Lagrange equations of an
underlying energy functional, possibly subject to a constraint, and domain decomposition can be applied directly at the level of the energy functional rather than at the level of the PDE (or it's weak formulation). The shift from an operator-based to an energy-based decomposition allows for the simultaneous consideration of problem classes that have so far been treated as fundamentally distinct. An energy-based approach also has the advantage that the energy functional itself provides the mechanism for combining previous iterates with local updates, which is conceptually simpler than explicitly imposing transmission conditions on subdomains.

The method is a two-level optimization scheme: at each iteration, local
subproblems are solved in parallel, followed by a global second-level
optimization. A first distinctive feature is that no local boundary conditions
are imposed. Instead, in the spirit of Maliassov~\cite{maliassov1998schwarz,heid2021gradient} but not restricted to eigenvalue problems, the previous global iterate is built into the local approximation spaces; since this function is the same for
all subproblems, it causes no computational burden. The second-level problem
then minimizes the energy over the span of the local updates together with the
past iterates. This is the second distinctive feature: previous solutions are
naturally embedded in the framework, similar in spirit to LOBPCG
\cite{knyazev2001toward} for eigenvalue problems. Retaining the last two iterates rather than only the
last accelerates convergence considerably for all tested model problems.

To the best of our knowledge, no existing framework provides a single,
energy-based formulation encompassing linear, nonlinear, source, and eigenvalue
problems simultaneously. We show that such a framework, the energy-minimizing
domain decomposition (EMDD) method, is feasible and develop its algorithmic realization on four model problems covering each class, and assess its practical viability through numerical experiments benchmarked against other DD-based methods. We expect the method to admit adaptation and fine-tuning to specific problems in many ways; this is not pursued here, but remarks on possible refinements are included throughout.

This paper is organized as follows. Section~\ref{sec:framework} presents the problem setting and the general framework. The four model problems and the discretization are introduced in Section~\ref{sec:modpb-and-disc}. In Section~\ref{sec:appl-mod-pb} we formulate the corresponding EMDD methods and analyze their complexity. Section~\ref{sec:numerics} reports numerical tests for all four model problems.

\section{Problem Setting and Numerical Method}
\label{sec:framework}
\noindent Let $V$ be a Hilbert space with inner product $(\cdot, \cdot)_V$ and let $\mathcal{M} \subset V$ denote a smooth embedded submanifold of $V$.
Typically, the space $V$ is a suitable Sobolev space or a finite-dimensional discretization thereof, while the manifold $\mathcal{M}$ is either $V$ itself or defined via some constraint. Regardless, we consider a general minimization problem of the form
\begin{align}\label{eq:2.1}
    u^* = \argmin_{v\in \mathcal{M}} \en(v),
\end{align}
for some $\mathscr{C}^2$-energy functional $\en: \mathcal{M} \to \R$. 

We assume that the minimization problem \eqref{eq:2.1} is well-posed in the sense that the energy functional $\en$ admits minimizers on $\mathcal{M}$ and that these minimizers are unique (possibly up to some trivial gauge invariance). The existence and uniqueness of local minimizers is typically demonstrated by imposing suitable weak lower semi-continuity and convexity conditions on the energy functional, but we shall not address these questions in further detail in the present section. Instead, following our presentation of the general numerical method, we will focus, from Section \ref{sec:2.2} onwards, on four classes of model problems, each of which involve energy functionals possessing unique (up to gauge invariance) minimizers.

In order to present our numerical algorithm, we assume that the ambient Hilbert space $V$ possesses a (not necessarily direct) decomposition
\begin{align}\label{eq:decomp}
    V = \sum_{i=1}^M V_i,
\end{align}
with each $V_i$ a Hilbert subspace of $V$. A standard example of such a decomposition-- widely used in the context of elliptic PDEs -- involves $V= H^1_0(\Omega)$ for an open, bounded set $\Omega \subset \R^d$ and $V_i = H_0^1(\Omega_i)$ where $\overset{M}{\underset{i=1}{\cup}} \Omega_i = \Omega$ is a (possibly overlapping) decomposition of~$\Omega$. Let us nevertheless emphasize that this setting is simply an example, and our proposed method is applicable to any decomposition of the form~\eqref{eq:decomp}.

The \emph{energy-minimizing domain decomposition} (EMDD) method is defined~as follows: Given the  history parameter $q\in\N=\{1, 2, \ldots\}$ and an initial guess $u^{(0)}\in \mathcal{M} \subset V$, generate iterates $u^{(k)},~ k\in \mathbb{N}$ by

\begin{enumerate}[label=(\arabic*)]
    \item[(i)] Introducing, for each $i=1, \ldots, M$, enriched local spaces 
    \begin{align}
    \label{eq:enriched_subspaces}
        \Vik := V_i + \vspan\{u^{(k-1)}\},
    \end{align}
    and solving, in parallel, the local optimization problems
    \begin{equation}
        \label{eq:EMDD-step1}
        y_i^{(k)}
        := 
        \argmin_{y_i \in \mathcal{M} \cap \Vik} \, \en(y_i).
    \end{equation}

    \item[(ii)] Introducing the coarse space 
    \begin{align*}
        \Vzk := \vspan\left\{ u^{(k-\tilde q)},\ldots,u^{(k-1)}, y_1^{(k)}, \ldots, y_m^{(k)}\right\} \quad \text{with}\quad \tilde q = \min(q,k),
    \end{align*}
    and solving a second-level optimization problem
    \begin{equation}
        \label{eq:EMDD-step2}
        u^{(k)} 
        :=
        \argmin_{u \in \mathcal{M} \cap \Vzk } \,
        \en(u).
    \end{equation}
\end{enumerate}

Consider the energy-minimizing domain decomposition algorithm defined via Equations \eqref{eq:EMDD-step1} and \eqref{eq:EMDD-step2}. Some important comments are now in order.

\begin{remark}[Structure of EMDD Local and Second-Level Minimization Problems]\label{rem:structure}
In the case when the global minimization problem \eqref{eq:2.1} is \emph{unconstrained}, i.e., $\mathcal{M}=V$, Equations \eqref{eq:EMDD-step1} and \eqref{eq:EMDD-step2} correspond to Galerkin approximations of the global problem on suitable subspaces, and thus inherit the structure of the minimization problem~\eqref{eq:2.1}.

The situation is less trivial when the global minimization problem \eqref{eq:2.1} is \emph{constrained}, i.e., $\mathcal{M} \subsetneq V$. The essential difficulty is that for a generic manifold $\mathcal{M} \subset V$ and an arbitrary subspace $W \subset V$, the set $\mathcal{M} \cap W$ may not be a submanifold of $\mathcal{M}$ or may even be empty. A sufficient condition to guarantee that $\mathcal{M}\cap W$ \emph{is} a submanifold of $\mathcal{M}$ is that $\mathcal{M}$ and $W$ intersect \emph{transversally}, i.e., $\forall x \in \mathcal{M}\cap W$, we have
\begin{align*}
    \mathcal{T}_x \mathcal{M} + W = V.
\end{align*}

Here, $\mathcal{T}_x\mathcal{M}\subset V$ denotes the tangent space of the manifold $\mathcal{M}$ at the point $x$. This condition will be the subject of further discussion in Propositions \ref{prop:fixed_points} and \ref{prop:assumptions}.
\end{remark}

\begin{remark}[Computational Complexity of the EMDD Algorithm]\label{rem:complex} 
A detailed study of the computational complexity of the EMDD for several model problems is presented in Sections \ref{sec:modpb-and-disc} and \ref{sec:appl-mod-pb}. Let us nevertheless point out that each of the $M$ local minimization problems \eqref{eq:EMDD-step1} can be solved in parallel and the second-level minimization problem is posed over the set $\Vzk \cap \mathcal{M}$ with $\Vzk$ of small dimension $\min(q, k)+M$. Provided, therefore, that the subspaces $\{V_j\}_{j=1}^M$ are \emph{localized} (see Section \ref{sec:2.4}) and that the structure of the global minimization problem \eqref{eq:2.1} is inherited by the second-level minimization \eqref{eq:EMDD-step2}, we can expect the EMDD algorithm to be computationally feasible.  
\end{remark}

\begin{remark}
The present formalism carries potential for the development of multiscale
methods. The local basis functions are not restricted to a particular
choice: multiscale finite element bases~\cite{hou1997multiscale}, bases generated by the Localized Orthogonal Decomposition method~\cite{maalqvist2014localization}, or bases obtained from local eigenvalue problems, as in GenEO~\cite{spillane2014abstract}, are equally admissible. It is not the purpose of this article, however, to propose fine-grained or specialized constructions, but rather to demonstrate the universality of the method using standard approximation spaces such as the finite element space. A structural difference with many other domain decomposition and multiscale concepts is that EMDD requires no local boundary conditions, since $u^{(k-1)}$ is contained in the local and coarse spaces.  We also note that a very similar two-level algorithm for high-dimensional optimization problem in the tensor train format has recently been proposed by one of the authors \cite{grigori2026additive}.
\end{remark}

The following proposition provides a partial answer to another natural question pertaining to the EMDD algorithm, namely, the relation between its fixed points and solutions of the minimization problem \eqref{eq:2.1}.

\begin{proposition}[Fixed Points of the EMDD Algorithm]\label{prop:fixed_points}
   Consider the minimization problem \eqref{eq:2.1}, the subspace decomposition \eqref{eq:decomp}, and the EMDD algorithm defined through Equations \eqref{eq:EMDD-step1}-\eqref{eq:EMDD-step2}.

   \begin{itemize}
       \item Assume that the minimization problem \eqref{eq:2.1} admits a unique minimizer $u_{\rm min}$. Then $u_{\rm min}$ is also a fixed point of the EMDD algorithm in the sense that
       \begin{align*}
\forall k\in \mathbb{N}\colon \qquad u^{(k-1)}= u_{\rm min} \implies u^{(k)}= u_{\rm min}.           
       \end{align*} 
    \item Conversely, let $u_{\rm fix}$ denote a fixed point of the EMDD algorithm in the sense that there exists $ k \in\mathbb{N}$ such that $u^{(k-1)}= u^{(k)}=u_{\rm fix}$,  
and assume the following:
\begin{enumerate}

    \item For each $i \in \{1, \ldots, M\}$, the local minimization problem
    \begin{align}\label{eq:fixedpoint_assum0}
        \argmin_{y_i \in \mathcal{M} \cap \Vik} \quad \text{with } \quad \Vik:= V_i + \text{span} \{u_{\rm fix}\} \qquad \text{is uniquely solvable}.
    \end{align}

    \item The decomposition \eqref{eq:decomp} satisfies $\forall i \in \{1, \ldots, M\}$ and $\forall y_i \in \mathcal{M}\cap \Vik$:
       \begin{align}\label{eq:fixedpoint_assum1}
       \mathcal{T}_{y_i} \mathcal{M} + \Vik = V, \qquad &(\text{local transversality condition})
       \end{align}       
    \item The subspace decomposition \eqref{eq:decomp} satisfies
       \begin{align}\label{eq:fixedpoint_assum2}
            \mathcal{T}_{u_{\rm fix}} \mathcal{M} = \sum_{i=1}^M \mathcal{T}_{u_{\rm fix}} \mathcal{M} \cap \Vik. \qquad &(\text{global span condition}) 
       \end{align}        
\end{enumerate}
Then $u_{\rm fix}$ is a critical point of the energy functional $\mathcal{E}$ on $\mathcal{M}$ in the sense that
       \begin{align*}
       \forall v \in \mathcal{T}_{u_{\rm fix}} \mathcal{M}\colon \qquad    \en'[u_{\rm fix}] (v)=0.
       \end{align*}
   \end{itemize}
\end{proposition}
\begin{proof}
Let  $u_{\rm min} \in \mathcal{M}\cap V$ be the unique solution of the minimization problem \eqref{eq:2.1}. Setting $u^{(k-1)}= u_{\rm min}$ in the EMDD algorithm for some $k \in \mathbb{N}$, we deduce that  $\Vik:= V_i + \text{span} \{u_{\rm min}\} ~ \subset V$ for all $i \in \{1, \ldots, M\}$.
It immediately follows that 
\begin{align*}
    \forall i \in \{1, \ldots, M\}\colon \argmin_{y_i \in \mathcal{M} \cap \Vik} \, \en(y_i) = u_{\rm min}.
\end{align*}
Consequently, $u_{\rm min}\in \Vzk$ and since $u_{\rm min}$ is, by definition the unique minimizer of the energy functional $\en$ on $\mathcal{M}$, we deduce that $u^{(k)} =\argmin_{u \in \mathcal{M} \cap \Vzk } \en(u)= u_{\rm min}$. Thus, $u_{\rm min}$ is a fixed point of the EMDD algorithm as claimed.

Conversely, let $u_{\rm fix}$ be a fixed point of the EMDD algorithm. We claim that 
\begin{align*}
    \forall i \in \{1, \ldots, M\}\colon \qquad \argmin_{y_i \in \mathcal{M} \cap \Vik} \, \en(y_i) = u_{\rm fix}.
\end{align*}
Indeed, suppose $\argmin_{y_j \in \mathcal{M} \cap \Vjk} \, \en(y_j) = y_j^{(k)} \neq u_{\rm fix}$ for some $j\in \{1, \ldots, M\}$.
Thanks to our unique solvability assumption \eqref{eq:fixedpoint_assum0}, we must then have $\en(y_j^{(k)}) < \en(u_{\rm fix})$.
Since $y_j^{(k)}$ is an element of the coarse space $\Vzk$, this implies that
\begin{align}\label{eq:lem_2}
    \min_{u \in \mathcal{M} \cap \Vzk } \,
        \en(u) \leq \en(y_j^{(k)}) < \en(u_{\rm fix}),
\end{align}
which, of course violates the fact that $u^{(k-1)}=u^{(k)}=u_{\rm fix}$, i.e., the fact that $u_{\rm fix}$ is a fixed point of the EMDD algorithm. We conclude that  $\forall i \in \{1, \ldots, M\}$, it holds that
\begin{align}\label{eq:lem_3}
    \argmin_{y_i \in \mathcal{M} \cap V_{i}^{(k)}} \, \en(y_i) = u_{\rm fix},
\end{align}
as claimed. Next, we note that the local transversality condition \eqref{eq:fixedpoint_assum1} implies that \cite[Chapter II, Prop. 2.4]{lang2002introduction} each set $\mathcal{M} \cap V_{i}^{(k)}$ is a submanifold of $\mathcal{V}$, and it holds that
\begin{align*}
    \mathcal{T}_{u_{\rm fix}} \big(\mathcal{M} \cap V_{i}^{(k)}  \big)  = \mathcal{T}_{u_{\rm fix}}\mathcal{M} \cap V_{i}^{(k)}. 
\end{align*}
Since $u_{\rm fix}$ is a solution to the minimization problem \eqref{eq:lem_3}, we therefore deduce that
\begin{align*}
\forall i \in \{1, \ldots, M\}, ~ \forall v_i \in \mathcal{T}_{u_{\rm fix}}\mathcal{M} \cap V_{i}^{(k)}\colon \qquad \en'[u_{\rm fix}](v_i)=0. \hspace{2cm}
\end{align*}
Appealing now to Assumption \eqref{eq:fixedpoint_assum2} yields that $\en'[u_{\rm fix}](v)=0$ for all $v \in \mathcal{T}_{u_{\rm fix}} \mathcal{M}$.
\end{proof}

\begin{remark}[Validity of Assumptions in Proposition \ref{prop:fixed_points}]\label{rem:assumptions}
     It is natural to ask when (if at all) the assumptions of Proposition \ref{prop:fixed_points}, i.e., the local unique solvability~\eqref{eq:fixedpoint_assum0}, the local transversality condition \eqref{eq:fixedpoint_assum1} and the spanning condition \eqref{eq:fixedpoint_assum2} are satisfied. 

    Concerning the question of unique solvability, we note that if $\mathcal{M}\equiv V$, then the local minimization problem \eqref{eq:EMDD-step2} can be viewed simply as a Galerkin approximation of the global minimization problem \eqref{eq:2.1}. Consequently, if Equation \eqref{eq:2.1} admits a unique minimizer, then the unique solvability condition \eqref{eq:fixedpoint_assum0} clearly holds. For similar reasons, the condition \eqref{eq:fixedpoint_assum0} also holds (up to sign) for linear eigenvalue problems, i.e., when $\mathcal{M}\subsetneq V$ is a unit sphere and $\en\colon \mathcal{M}\to \R$ is a quadratic energy functional. On the other hand, it is not clear if condition \eqref{eq:fixedpoint_assum0} holds more generally, for instance, for \emph{nonlinear} eigenvalue problems. In the numerical experiments presented in Section \ref{sec:numerics}, we did not observe issues of non-uniqueness (discounting trivial gauge freedoms) for the nonlinear eigenvalue problem.

    Concerning the local transversality and global spanning assumptions \eqref{eq:fixedpoint_assum1}-\eqref{eq:fixedpoint_assum2}, we similarly note that these are trivially satisfied in the case $\mathcal{M}\equiv V$. Additionally, the following proposition demonstrates that these conditions are satisfied also when the constraint set $\mathcal{M}$ is a unit sphere, which is, for instance, the case for second-order elliptic eigenvalue problems.
\end{remark}

\begin{proposition}[Validity of Assumptions \eqref{eq:fixedpoint_assum1} and \eqref{eq:fixedpoint_assum2} for Unit Spheres]\label{prop:assumptions}
Assume that the Hilbert space $(V, \Vert \cdot \Vert_V)$ is continuously embedded in another Hilbert space $(H, \Vert \cdot \Vert)_H$, i.e., $V\subset H$ and there exists $C>0$ such that $\Vert v \Vert_H \leq C \Vert v \Vert_V$ for all $v \in V$. Suppose that the manifold $\mathcal{M} \subset V$ is defined as $\mathcal{M} = \left\{v \in V\colon \hspace{2mm} \Vert v \Vert_{H}=1\right\}$, let $\sum_{i=1}^M V_i =V$ be a subspace decomposition of $V$, let $w \in \mathcal{M}$ be arbitrary, and for any $i \in \{1, \ldots, M\}$ define $V_{i, w}= V_i + \text{\rm span}\{w\}$. Then, for all $i \in \{1, \ldots, M\}$ and  $v_i \in \mathcal{M}\cap V_{i, w}$ it holds that
\begin{align}\label{eq:prop_2}
     \quad \mathcal{T}_{v_i} \mathcal{M} + V_{i, w} = V, \qquad \text{and}\qquad \mathcal{T}_{w} \mathcal{M} = \sum_{i=1}^M \mathcal{T}_{w} \mathcal{M} \cap V_{i, w}.
\end{align}
\end{proposition}
\begin{proof}
    We first prove the identity appearing on the left-hand side of Equation \eqref{eq:prop_2}. Obviously, $\mathcal{T}_{v_i} \mathcal{M} + V_{i, w} \subset V$ so let us demonstrate the converse inclusion. To this end, let $i \in \{1, \ldots, M\}$ and $v_i \in \mathcal{M} \cap V_{i, w}$ be given. Clearly, we have $\mathcal{T}_{v_i}\mathcal{M}= \left\{u \in V \colon (u, v_i)_{H}=0\right\}$. Let now $v \in V$ be arbitrary and consider the decomposition
\begin{align*}
    v= \big(v- (v_i, v)_{H}\; v_i\big) + \big((v_i, v)_{H}\;v_i\big) =: v^\perp + v^{\parallel}.
\end{align*}
Since $v^{\perp} \in \mathcal{T}_{v_i}\mathcal{M}$ and $v^{\parallel} \in V_{i, w}$, we clearly have $v \in \mathcal{T}_{v_i} \mathcal{M} + V_{i, w}$ so that the first identity in Equation~\eqref{eq:prop_2} is indeed true.
We now turn our attention to the second identity in Equation \eqref{eq:prop_2}. Since we obviously have 
\begin{align*}
 \sum_{i=1}^M \mathcal{T}_{w} \mathcal{M} \cap V_{i, w} \subset \mathcal{T}_{w} \mathcal{M}, 
\end{align*}
let us demonstrate the converse inclusion. To this end, let us observe that 
\begin{align*}
    \mathcal{T}_{w}\mathcal{M}= \left\{u \in V \colon (u, w)_{H}=0\right\} 
    \quad \text{and} \quad \mathcal{T}_{w}\mathcal{M} \cap V_{i, w}= \left\{u_i \in V_{i, w} \colon (u, w)_{H}=0\right\}.
\end{align*}
Consider now $v \in \mathcal{T}_{w}\mathcal{M}$. We can clearly write $v= \sum_{i=1}^M v_i$ with each $v_i \in V_i$. Additionally, each of these functions $v_i$ can be decomposed as 
\begin{align*}
    v_i = \big(v_i - (w, v_i)_H \; w\big) + \big((w, v_i)_H \; w\big)=:v_i^\perp + v_i^\parallel, \quad  \text{with } ~v_i^\perp\in \mathcal{T}_{w}\mathcal{M} \cap V_{i, w}.
\end{align*}
It now suffices to notice that we can write
\begin{align*}
    v=\sum_{i=1}^M v_i^\perp + \sum_{i=1}^M v_i^\parallel= \sum_{i=1}^M v_i^\perp  +\big(w,  \sum_{i=1}^M v_i\big)_H \; w&=\sum_{i=1}^M v_i^\perp  +\big(w, v\big)_H \; w
    =\sum_{i=1}^M v_i^\perp,
\end{align*}
with the last step following from the fact that $v \in \mathcal{T}_{w}\mathcal{M}$. Thus, $v \in \sum_{i=1}^M \mathcal{T}_{w} \mathcal{M} \cap V_{i, w}$ as required, and this completes the proof of the second identity.
\end{proof}

\section{Model problems, discretization, and complexity}
\label{sec:modpb-and-disc}
Four representative model problems will be introduced in this section together with a general discretization framework based on abstract assumptions that are all fulfilled in a finite element context.

\subsection{Presentation of model problems}\label{sec:2.2}
Throughout this section and in the sequel, we denote by
\begin{itemize}
    \item $\Omega \subset \R^d$ for $d \in \{1, 2, 3\}$ a bounded, Lipschitz domain and define $V:= \rH_0^1(\Omega)$;
    \item $a:V\times V \to \R$, a symmetric, continuous, coercive bilinear form;
    \item $f:V\to \R$, a bounded linear functional;
   \item $G_{\rm src}\in C^1\big([0, \infty)\big)$, a real-valued function that satisfies
    \begin{align} \label{eq:irregular}
       [0, \infty) \ni t \mapsto G_{\rm src}'(t) \sqrt{t} \text{ is  non-decreasing}, \qquad \text{and} \\[0.5em] \nonumber
        \exists q_{\rm src}\in [0, 2], C>0, \forall t\in (0, \infty) \colon \qquad |G_{\rm src}'(t)|&\leq C(1+  t^{q_{\rm src}}).
    \end{align}
    \item $G_{\rm eig}\in C^1\big([0, \infty)\big)\cap C^2\big((0, \infty)\big)$,  real-valued function that satisfies
    \begin{align}\label{eq:irregular_eig}
        G_{\rm eig}''> 0 \text{ on } (0, \infty), \quad G_{\rm eig}''(t)t \text{ is bounded around }t&=0, \qquad \text{and}    \\[0.5em] \nonumber
        \exists q_{\rm eig}\in [0, 2), C>0, \forall t \in (0, \infty) \colon \qquad |G_{\rm eig}'(t)|&\leq C(1+ t^{q_{\rm eig}}).
    \end{align}    
\end{itemize}
The model problems that we will focus on in the sequel are now defined as follows:
\begin{enumerate}
    \item \textbf{Coercive linear source problems.}
    We set $\mathcal{M}=V:= \rH_0^1(\Omega)$ and consider
    \begin{equation}
        \label{eq:linear_sys}
        \forall u \in \mathcal{M}\colon \qquad       \en(u) := \tfrac12 a(u,u) - f(u).
    \end{equation}
    The Euler-Lagrange equation associated with Equation \eqref{eq:linear_sys} is given by
    \[
        a(u^*,v) = f(v) \qquad\forall v\in V.
    \]
    It follows from the Lax-Milgram lemma that Equation \eqref{eq:linear_sys} admits a unique minimizer which solves the above Euler-Lagrange equations.\vspace{3mm}

    \item \textbf{Monotone semi-linear source problems.}
    We set $\mathcal{M}=V:= \rH_0^1(\Omega)$ and consider
    \begin{equation}
        \label{eq:nonlinear_sys}
        \forall u \in \mathcal{M}\colon \qquad       \en(u) := \tfrac12 a(u, u) + \frac12\int_{\Omega} G_{\rm src}(u^2(\bold{x}))\; d \bold{x}- f(u).
    \end{equation}
    The Euler-Lagrange equation associated with Equation \eqref{eq:nonlinear_sys} is given by
    \begin{align*}
        a(u^*,v) + \int_{\Omega} G_{\rm src}'(u^*(\bold{x})^2)u^*(\bold{x})v(\bold{x})\; d \bold{x}= f(v) \qquad\forall v\in V.
    \end{align*}
    Thanks to the sufficient conditions \eqref{eq:irregular}, we can deduce from the Browder-Minty theorem that Equation \eqref{eq:nonlinear_sys} admits a unique minimizer which solves the above Euler-Lagrange equations. 
    
    \item \textbf{Linear eigenvalue problems.}
    We take   $\mathcal{M} := \big\{ v \in V \colon \; \Vert v \Vert_{L^2(\Omega)}^2 =1\big\}$,
    and consider
    \begin{equation}
        \label{eq:linear_eig}
        \forall u \in \mathcal{M}\colon \qquad      \en(u) :=a(u,u).
    \end{equation}
    Here, the coercivity condition on $a$ can be relaxed to the requirement that the bilinear form satisfies a Gårding inequality.
    The Euler-Lagrange equation associated with Equation \eqref{eq:linear_eig} is given by
    \begin{align*}
        a(u^*,v) = \lambda (u^*, v)_{L^2(\Omega)} \qquad\forall v\in V, \quad \text{with } \Vert u^*\Vert_{L^2(\Omega)}=1.
    \end{align*}
    Clearly, Equation \eqref{eq:linear_eig} admits minimizers and these minimizers correspond to the ground-state (lowest) eigenfunctions of the above linear eigenvalue problem. We assume that this eigenvalue is non-degenerate.\vspace{3mm}
    
    \item \textbf{Nonlinear eigenvalue problems.}
    We take again $\mathcal{M}:= \big\{ v \in V \colon \; \Vert v \Vert_{L^2(\Omega)}^2 =1\big\}$, and consider
    \begin{equation}
        \label{eq:nonlinear_eig}
        \forall u \in \mathcal{M}\colon \qquad    {\en}(u) := \tfrac{1}{2} a(u, u) + \frac12\int_{\Omega} G_{\rm eig}(u^2(\bold{x}))\; d \bold{x}.
    \end{equation}
    As before, the coercivity condition on $a$ can be replaced with a Gårding inequality. Regardless, the Euler-Lagrange equations associated with Equation \eqref{eq:nonlinear_eig} are given by finding $u^*\in\mathcal M$ such that
    \begin{align*}
        a(u^*,v) + \int_{\Omega} G_{\rm eig}'(u^*(\bold{x})^2)u^*(\bold{x})v(\bold{x})d \bold{x} = \lambda (u^*, v)_{L^2(\Omega)} \qquad\forall v\in V.
        \remove{\quad \text{with } \Vert u^*\Vert_{L^2(\Omega)}=1.}
    \end{align*}
  Under the assumptions \eqref{eq:irregular_eig}, it has been proven in \cite{cances2010numerical}, that Equation \eqref{eq:nonlinear_eig} admits a unique (up to sign) minimizer.
\end{enumerate}
\medskip

In view of Remark~\ref{rem:assumptions} and Proposition~\ref{prop:assumptions}, we see that--with the possible exception of unique solvability (excluding gauge freedoms) of the \emph{discretized} nonlinear eigenvalue problem--the model problems  \eqref{eq:linear_sys}-\eqref{eq:nonlinear_eig} satisfy the assumptions of Proposition \ref{prop:fixed_points}. For each of these problems, we will describe in more detail the structure of the EMDD local and second-level minimization problems and discuss the computational complexity of the EMDD algorithm. This is the subject of Section \ref{sec:appl-mod-pb}. Prior to doing so however, we introduce, in the forthcoming Sections \ref{sec:2.3} and \ref{sec:2.4},  some common notation and definitions that will assist our exposition.

\subsection{General framework for Galerkin discretization}\label{sec:2.3}

In the numerical practice, the model problem~\eqref{eq:2.1} with energies defined as in \eqref{eq:linear_sys}-\eqref{eq:nonlinear_eig}, which are a priori formulated on $V\equiv H_0^1(\Omega)$, are discretized and solved on an $N$-dimensional subspace $\Vapp \subset H_0^1(\Omega)$. Thanks to the variational principle, this amounts to replacing the space $V$ with $\Vapp$ in our earlier exposition but leaves the structural setting of the model problems intact. Thus, we still assume that 
\begin{itemize}
    \item The approximation space $\Vapp$ possesses a subspace decomposition $\Vapp= \sum_{i=1}^M V_i$ (c.f., Equation \eqref{eq:decomp}) but now with each subspace $V_i$ of dimension $N_i \in \mathbb{N}$. \smallskip
    \item  Additionally, we denote by $\{\phi_{n}\}_{n=1}^N$ a basis for $\Vapp$ and by $\{\phi^{i}_{n}\}_{n=1}^{N_i}$ a basis for any subspace $V_i, ~ i \in \{1, \ldots, M\}$. 
\end{itemize}

\vspace{2mm}

Our goal now is to describe in more detail the discretized bilinear form $a$, nonlinear mappings $G_{\rm src}$ and $G_{\rm eig}$ and right-hand-side vector $f$ that result from this finite-dimensional setting, and are used to formulate Equations \eqref{eq:EMDD-step1}-\eqref{eq:EMDD-step2} in the EMDD algorithm. Since the subsequent exposition requires considerable notation, we briefly list some conventions we adopt in the use of indices. 
\begin{convention}[Index Conventions]
    In the remainder of this article:
\begin{itemize}
    \item We use $i,j \in \{1,\ldots,M\}$ for indices that label either subspaces or subdomains;

    \item We use $k,\ell$ for indices that label the iteration number of the EMDD algorithm;

    \item We use $m,n \in \{1,\ldots,N\}$ for indices that label basis functions;
    
    \item We use $N_i$ to denote the number of basis functions in the $i^{\rm th}$ subspace.
\end{itemize}
\end{convention}

\begin{convention}
    In the sequel, we will frequently consider vectors and matrices whose definition depends on the EMDD iteration index $k$ via the previous EMDD iterates $u^{(\ell)}$, for $\ell=\tilde q,\ldots, k-1$. For the sake of a lighter notation, this dependence will frequently not be indicated explicitly via superscripts or subscripts.
\end{convention}

\begin{definition}[discretized operators for the local minimization steps]\label{def:loc_matrix}
Consider the EMDD algorithm defined through Equations \eqref{eq:EMDD-step1}-\eqref{eq:EMDD-step2}.  Let $i \in \{1, \ldots, M\}$, $k \in \{1,\ldots,\}$, and denote $\phi^{i}_{N_i+1}:= u^{(k-1)}$.

At iteration $k$ of the EMDD algorithm, we define matrices $\Ai, \Si \in \R^{(N_i+1) \times (N_i+1)}$: 
\begin{align}
    \label{eq:AS_matrix_def}
    \Ai_{mn} := a(\phi^{i}_{m}, \phi^{i}_{n})
    \qquad\textrm{and}\qquad
    \Si_{mn} := (\phi^{i}_{m}, \phi^{i}_{n})_{L^2(\Omega)},
\end{align}
for all $m,n \in \{1, \ldots, N_i+1\}$.
Additionally, given any $u \in \Vik$ that is represented by ${\bf u}\in \R^{N_i+1}$, we define  $\Ni({\bf u}), \Hi({\bf u}) \in \R^{(N_i+1) \times (N_i+1)}$ as the matrices with the property that for all $m,n \in \{1, \ldots, N_i+1\}$, $G\in \{G_{\rm src},G_{\rm eig}\}$,
\begin{align}
    \label{eq:AS_nonlinear_def}
     \Ni_{mn}({{\bf u}}) &= \int_{\Omega}G'(u^2(\bold{x})) \phi^{i}_m(\bold{x}) \phi^{i}_n(\bold{x})\; d \bold{x},
    \\
    \label{eq:H_local_def}
    \Hi_{mn}({{\bf u}}) &= \int_{\Omega} \Big(G'(u^2(\bold{x})) + 2 \, G''(u^2(\bold{x}))u^2(\bold{x})\Big) \phi^{i}_m(\bold{x}) \phi^{i}_n(\bold{x})\; d \bold{x}.
\end{align}
Finally, we define the vector $\fiv\in \R^{N_i+1}$ as
\begin{align}
    \label{eq:fN_vector_def_loc}
    \forall m\in \{1,\ldots,N_i +1\}\colon \qquad    \fiv_m= f(\phi^{i}_{m}).
\end{align}
\end{definition}

\vspace{2mm}

\begin{definition}[discretized operators for the second-level minimization step]\label{def:global_matrix}
Consider the EMDD algorithm defined through Equations \eqref{eq:EMDD-step1}--\eqref{eq:EMDD-step2}. Let $k \in \{1, \ldots, \}$ and let the set $\{\chi_{n}\}_{n=1}^{\tilde q  +M}$ be defined as
\begin{align*}
    \chi_n := 
    \begin{cases}
       u^{(k-n)} \quad &\text{for }~ n =1, \ldots, \tilde q
       \\
        y_{n -\tilde q}^{(k)} \quad &\text{for }~ n =\tilde q +1, \ldots, \tilde q  +M,
    \end{cases}
\end{align*}

At iteration $k$ of the EMDD algorithm, we define matrices $\Az, \Sz \in \R^{(\tilde q  +M) \times (\tilde q  +M)}$:
\begin{align}
    \label{eq:A0S0_matrix_def} 
    \Az_{mn}= a(\chi_m, \chi_n)
    \qquad\textrm{and}\qquad
    \Sz_{mn}= (\chi_m, \chi_n)_{L^2(\Omega)},
\end{align}
for all $m,n \in \{1,\ldots,\tilde q+M\}$.
Additionally, given any $u \in V^k_0$ that is represented by ${\bf u}\in \R^{\tilde q+M}$, we define $\Nz({\bf u}), \Hz({\bf u}) \in \R^{(\tilde q+M) \times (\tilde q+M)}$ as the matrices with the property that for all $m,n \in \{1, \ldots, \tilde q+M\}$, $G\in \{G_{\rm src},G_{\rm eig}\}$,
\begin{align}\label{eq:N0_matrix_def} 
    \Nz_{mn}({{\bf u}}) &= \int_{\Omega}G'(u^2(\bold{x})) \chi_m(\bold{x}) \chi_n(\bold{x})\; d \bold{x}, 
    \\
    \label{eq:H_coarse_def}
    \Hz_{mn}({{\bf u}}) &= \int_{\Omega} \Big(G'(u^2(\bold{x})) + 2 \, G''(u^2(\bold{x}))u^2(\bold{x})\Big) \chi_m(\bold{x}) \chi_n(\bold{x})\; d \bold{x}.
\end{align}
Finally, we define the vector $\fz\in \R^{\tilde q+M}$ as
\begin{align}
    \label{eq:fN_vector_def_global}
    \forall m \in \{1,\ldots,\tilde q+M\}\colon \qquad       \fz_m= f(\chi_m)
\end{align}
\end{definition}

\vspace{2mm}

Having introduced these Euclidean vectors and matrices and recalling the local and global basis functions $\{\phi^i_n\}_{\substack{i=1, \ldots,M\\ n=1, \ldots, N_i}}$ and $\{\phi_n\}_{n=1}^N$ respectively, we can represent the solutions to the two optimizations problems \eqref{eq:EMDD-step1}--\eqref{eq:EMDD-step2}
\begin{align}
    \label{eq:sol-rep-1}
    y_i^{(k)} &= \sum_{n=1}^{N_i+1} \yi_n \phi^{i}_{n}
    =
    \sum_{n=1}^{N_i} \yi_n \phi^{i}_{n}
    +
    \yi_{N_i+1} u^{(k-1)}
    \\
    \label{eq:sol-rep-2}
    u^{(k)} &= \sum_{n=1}^{\tilde q+M} {\bf{u}}^{(k)}_n \chi_n
    =
    \sum_{n=1}^{\tilde q} {\bf{u}}^{(k)}_n u^{(k-n)} + \sum_{j=1}^{M} {\bf{u}}^{(k)}_{\tilde q + j} y_j^{(k)}     
    =
    \sum_{n=1}^N {\bf{U}}_n^{(k)} \phi_n,
\end{align}
with $\yi\in \R^{N_i+1}$, ${\bf{u}}^{(k)}\in \R^{\tilde q + M}$ and ${\bf{U}}^{(k)}\in \R^{N}$.

\begin{remark}
    In practice, it is advised to orthogonalize the basis of the coarse space $\Vzk$ by means of a QR-decomposition to ensure numerical stability and possibly to remove (near) linear dependency. This becomes highly relevant near convergence when a $y_i^k$ consists of $u^{(k-1)}$ plus a very small correction.
    However, we omit this numerical trick here for simple presentation of the method. 
\end{remark}

\vspace{2mm} 

\subsection{General framework for complexity analysis}\label{sec:2.4}

Our purpose in this section is to set up a general framework allowing us to study the computational cost of the EMDD algorithm as applied to the four classes of model problems \eqref{eq:linear_sys}-\eqref{eq:nonlinear_eig}. The reason such a complexity analysis is important is that a naive study of the EMDD algorithm \eqref{eq:EMDD-step1}-\eqref{eq:EMDD-step2}, in particular the use of global history vectors, might suggest that the cost of the algorithm scales as $\mathcal{O}(MN)$. We will show that--thanks to some tricks--the actual scaling of EMDD is considerably better, and does not include an $\mathcal{O}(MN)$ term. Let us also remark that the final computational scaling that we provide for each model problem (see Section \ref{sec:appl-mod-pb}) will be based on the use of iterative solvers for each step of the EMDD algorithm. Depending on the number of subdomains and the dimensions of the local subspaces, it might be advisable to use instead direct solvers, in which case, the final scaling can easily be deduced from the results of this section.

To begin with, we must specify more precisely the computational cost of constructing quantities of the form
\[
    \int_\Omega c(\bold{x}) \mathcal L(u)(\bold{x}) \mathcal L(v)(\bold{x}) \; d\bold{x}, \qquad \text{and} \qquad f(u). 
\]
Here, $c$ is a real-valued function, $f\colon V_{\rm approx}\rightarrow \R$ denotes some given linear form (appearing, e.g., in Equations \eqref{eq:linear_sys} and \eqref{eq:nonlinear_sys}), $u, v \in V_{\rm approx}$ and $\mathcal L$ denotes a suitable operator on $V_{\rm approx}$ such as the gradient or the identity.

In this framework, and recalling that we denote by $\{\phi_n\}_{n=1}^N$ and $\{\phi_n^{i}\}_{n=1}^{N_i}$ bases for $V_{\rm approx}$ and $V_i, ~ i \in \{1, \ldots, N_i \}$, complexity statements are made primarily with respect to $M$, $N$ and $N_i$.  Other quantities are mentioned explicitly if important. Moreover, we assume for the purpose of the present contribution that

\begin{itemize}[leftmargin=0.9cm]
    \item[(C1)] \textbf{Function evaluation.}
For any function $v$ in $\Vapp$, $V_i$, $\Vik$ or $\Vzk$ whose degrees of freedom with respect to the corresponding basis are explicitly given, a single point evaluation $v({\bf x})$ or $v^2({\bf x})$ costs $\comp{1}$.

    \item[(C2)] 
    \textbf{Linear form.}
    Each quantity $f(\phi_n)$ and $f(\phi_n^{i})$ can be computed in $\mathcal O(1)$.
    \item[(C3)] 
    \textbf{Bilinear form I.}
    For any $u\in\Vapp$ whose degrees of freedom with respect to the basis $\{\phi_n\}_{n \in \mathbb{N}}$ are explicitly given, the computation of quantities of the form
    \[
    \int_\Omega c(\bold{x}) \mathcal L( \phi_n)(\bold{x}) \mathcal L(u)(\bold{x}) \; d\bold{x} 
    \qquad\mbox{and}\qquad
    \int_\Omega c(\bold{x}) \mathcal L( \phi_n^{i})(\bold{x}) \mathcal L(u)(\bold{x}) \; d\bold{x} 
\]
scales for all $i=1,\ldots,M$ as $\mathcal O(N_c)$ where $N_c$ is the cost of evaluating $c({\bf{x}})$ at a single point ${\bf{x}} \in \Omega$.

\item[(C4)]
\textbf{Bilinear form II.}
For any $u, v \in V_{\rm approx}$ and any $u_i, v_i \in V_i$ whose degrees of freedom with respect to the basis $\{\phi_n\}_{n \in \mathbb{N}}$ and $\{\phi^i_n\}_{n \in \mathbb{N}}$ respectively are explicitly given, the computation of quantities of the form
    \begin{align*}
        \int_\Omega c(\bold{x}) \mathcal L( u)(\bold{x}) \mathcal L(v)(\bold{x}) \; d\bold{x}  \quad \text{and} \quad \int_\Omega c(\bold{x}) \mathcal L( u_i)(\bold{x}) \mathcal L(v_i)(\bold{x}) \; d\bold{x}, 
    \end{align*}
scales as $\mathcal O(N_c N)$ and $\mathcal{O}(N_c N_i)$ respectively. As before, $N_c$ denotes the cost of evaluating $c({\bf{x}})$ at a single point ${\bf{x}} \in \Omega$.

\item[(C5)] 
\textbf{Intersection of subspace decomposition.}
For each $i=1,\ldots,M$ the number of subspaces $j=1,\ldots,M$ such that $ V_i \cap V_j \not = \emptyset$ is at most~$L$ with $L=\comp{1}$. In terms of domain decomposition, $L$ denotes the maximal overlap multiplicity of subdomains~$\Omega_i$.

\item[(C6)] 
\textbf{Identification of dominant parameters.}
There holds $q\le \min(M,\Nloc)$ for $\Nloc:=\max_{i=1,\ldots,M} N_i$.

\end{itemize}

\begin{remark}
Note that Assumptions (C1)-(C4) are naturally satisfied for finite element discretizations for local PDEs, thus with local operators $\mathcal L$.
Assumption (C5) is standard in the context of domain decomposition and sometimes referred to the maximal overlap condition. 
Finally, assumption (C6) is realistic since, as we shall see in the upcoming numerical tests, $q=2$ seems to be the favorite choice for $q$.
\end{remark}

Equipped with the above conventions and assumptions, we can now discuss the computational complexity related to operations involving the vectors $\fiv$, the matrices $\Ai$ and the nonlinear operators $\Ni({\bf U})$ in the local and second-level optimization steps of the EMDD method.

\subsubsection{The vectors \texorpdfstring{$\fiv$}{fi} and \texorpdfstring{$\fz$}{fz}}
Recalling the definition
\begin{equation}
    \fiv_n = f(\phi^{i}_{n}) = 
    \begin{cases}
       f\big(\phi^{i}_{n}\big) & \text{for } n \in \{1, \ldots, N_i\},
       \\[1em]
       f\big(u^{(k-1)}\big) & \text{for } n = N_i+1,
   \end{cases}
\end{equation}
we observe that we need assemble $f\big(u^{(k-1)}\big)$ only once per iteration $k$ and that this construction can be shared among all subdomains $i=1,\ldots,M$. Relying on (C2), we conclude that the cost of assembling all $M$ vectors ${\bf{f}}^{1},\ldots,{\bf{f}}^{M}$ scales as $\mathcal{O}(N+M\Nloc)$ at the first iteration $k=1$. The cost of subsequent updates related to the computation of $f\big(u^{(k-1)}\big)$ for iterations $k>1$ scales as $\mathcal O(N)$.

\medskip

Next, recalling the definition
\begin{align}
    \fz_n &= f(\chi_n) 
    =
    f\big(u^{(k-n)}\big) 
    \\
    \fz_{\tilde q+n} &= f(\chi_{\tilde q+n}) 
    = f\big(y^{(k)}_{n}\big) 
    = \sum_{m=1}^{N_{n}} {\bf{y}}^{n}_m f\big(\phi^{n}_{m}\big)
        +
        {\bf{y}}^{n}_{N_{n}+1} f\big(u^{(k-1)}\big)
\end{align}
for $n \in \{1, \ldots, \tilde q\}$, $n \in \{ 1,\ldots,M\}$, and assuming that $f\big(u^{(k-j)}\big)$, $j=1,\ldots,\tilde q$, is precomputed, we see that the cost of assembling the vector $\fz$ scales as $\mathcal{O}(M N_{\rm loc})$.

\subsubsection{The matrices \texorpdfstring{$\Ai$}{Ai} and \texorpdfstring{$\Az$}{Az}}
\label{sssec:matrices}
We assume that at iteration $k$ of the EMDD method, the coefficients $\tilde{\bf A}_{\ell\ell'}=a(u^{(\ell)},u^{(\ell')})$, for $\ell,\ell'\in \{k-\tilde q,\ldots,k-1\}$ are available. Indeed, while the cost of a brute-force calculation for each $\tilde{\bf A}_{\ell\ell'}$ scales, by assumption (C4), as $\mathcal O (N)$, one can use the variational framework to compute such terms cheaply, as we discuss in Section \ref{sec:appl-mod-pb}.

\medskip

Equipped with this assumption, we now turn our attention to the local minimization problems in the EMDD algorithm. In order to solve these local problems involving $\Ai \in \R^{(N_i +1) \times (N_i+1)}$, it is typical to employ matrix-free methods such as Krylov solvers. We therefore study the cost of a single matrix vector product of the form $\bold{v} \mapsto \Ai \bold{v}$ defined through Equation \eqref{eq:AS_matrix_def}.

Recall the definition of the basis functions $\{\phi^{i}_j\}_{j=1}^{N_i+1}$ and observe that by definition
\begin{align*}
   \big(\Ai \bold{v}\big)_{n} =  
   \begin{cases}
   a\big(\phi^{i}_{n},  y^{(i)}\big)+{\bf v}_{N_i+1} a\big(\phi^{i}_{n}, u^{(k-1)}\big) 
   & \text{for } n \in \{1, \ldots, N_i\}\\[1em]
   \sum_{m = 1}^{N_i} {\bf v}_m
    a\big(u^{(k-1)}, \phi^{i}_{m}\big)+{\bf v}_{N_i+1} \tilde{\bf A}_{11}
   & \text{for } n = N_i+1,
   \end{cases}
\end{align*}
where $y^{i} = \sum_{m = 1}^{N_i} {\bf v}_m\phi^{i}_{m} \in V_i$ is known.
Since it is supposed that $\tilde{\bf A}_{(k-1)(k-1)} = a\big(u^{(k-1)}, u^{(k-1)}\big)$ is precomputed, we can appeal to the locality of the bilinear form $a$ via (C3) to deduce that the cost of such a matrix vector product scales as $\mathcal{O}(N_i)$.

\medskip

Turning now to the second-level minimization step in the EMDD algorithm, we first observe that the history parameter $q \in \mathbb{N}$ is typically very small compared to ${\rm dim} \Vapp$. 
If the number of subdomains $M\in \mathbb{N}$ is also very small compared to ${\rm dim} \Vapp$, one can employ dense linear algebra algorithms in this step. We thus study the cost of assembling the matrix
$\Az  \in \R^{(\tilde q + M) \times (\tilde q + M)}$ defined through Equation~\eqref{eq:A0S0_matrix_def}.
To this end, and for $y^{i} := \sum_{n = 1}^{N_i} {\bf y}_n^{i} \phi^{i}_{n} \in V_i$, it is efficient to precompute ${a}_i:=a\big(u^{(k-1)},y^{i}\big)$ for each $i=1,\ldots,M$ in $\comp{M\Nloc}$ using once again the locality of the bilinear form~$a$ via (C3).
Then, recalling the definition of the basis functions $\{\chi_n\}_{n=1}^{\tilde q + M}$, we see that each entry of the matrix $\Az$ belongs to one of three cases:
\begin{description}
    \item[Case One] We must compute a term of the form 
    \begin{align*}
     a(u^{(k-m)}, u^{(k-n)}) = \tilde{\bf A}_{(k-m)(k-n)} \quad \text{for some } m,n \in \{1,\ldots,\tilde{q}\}.
\end{align*}
By the previous assumption, these terms are stored in memory and the cost of computing such a term is therefore $\mathcal{O}(1)$.

 \item[Case Two] We must compute a term of the form 
    \begin{align*}
    a(u^{(k-m)}, y^{k}_i) \quad \text{for some }  m \in \{1,\ldots,\tilde{q}\}, ~\text{ and }~i \in \{1, \ldots, M\}.
\end{align*}
Making use of the basis expansion of $y^{k}_i$ given by \eqref{eq:sol-rep-1}, we see that
\begin{align*}
    a(u^{(k-m)}, y^{k}_i)
    = 
    \sum_{n=1}^{N_i} {\bf y}_n^{i}  a\big(u^{(k-m)}, \phi_n^{i}\big) + {\bf y}_{N_i+1}^{i} \tilde{\bf A}_{1m}
    =    
    a_i + {\bf y}_{N_i+1}^{i} \tilde{\bf A}_{1m}
\end{align*}
Both terms are precomputed, and the complexity is $\comp{1}$. 
\item[Case Three] We must compute a term of the form 
    \begin{align*}
        a(y^{k}_j, y^{k}_i)
        =& 
        a\big(y^{j},y^{i}\big) 
        + {\bf y}_{N_j+1}^{j} a_i 
        + {\bf y}_{N_i+1}^{i} a_j
        + {\bf y}_{N_j+1}^{j} {\bf y}_{N_i+1}^{i}  \tilde{\bf A}_{11}.
    \end{align*}
    Note that the cost of computing $a\big(y^{j},y^{i}\big)$ (and hence of computing the full term above) scales as $\comp{\Nloc}$ but that $a(y_j, y_i)=0$ if $V_i\cap V_j = \emptyset$.
\end{description}
All told, the computational cost of assembling the matrix $\Az $ scales as $\Comp{M\Nloc + q^2 + q M  + M\Nloc}$\} using (C5), which can be simplified to $\comp{M\Nloc}$ using (C6). 

Once the matrix $\Az$ has been assembled, the cost of computing the solution to linear systems involving $\Az$ scales as $\Comp{M^3}$. If instead of dense linear algebra, a Krylov-type solver is used, a matrix-vector product 
\begin{align*}
   \big(\Az \bold{v}\big)_{m} &=  
   \sum_{n=1}^{\tilde q}  \tilde{\bf A}_{mn}{\bf v}_n
   +
   \sum_{i=1}^{M}
   {\bf v}_{\tilde q+i} a\big(y^{k}_{i}, u^{(k-m)}\big) 
   && \text{for } m \in \{1, \ldots, \tilde q\},
   \\
   \big(\Az \bold{v}\big)_{\tilde q+ i}
   &= 
   \sum_{n=1}^{\tilde q} {\bf v}_n a\big(u^{(k-n)},  y^{k}_{i}\big)
   +
   \sum_{j=1}^{M}
   {\bf v}_{\tilde q+j} a\big(y^{k}_{j}, y^{k}_{i}\big) 
   && \text{for } i \in\{1,\ldots,M\},
\end{align*}
has to be computed. As before, we precompute ${a}_i:=a\big(u^{(k-1)},y^{i}\big)$ for each $i=1,\ldots,M$ in $\comp{M\Nloc}$. We also build the sums $c_{xy}=\sum_{j=1}^{M} {\bf v}_{\tilde q+j}{\bf y}_{N_j+1}^{j}$ and $a_x=\sum_{j=1}^{M} {\bf v}_{\tilde q+j}{a}_j$ in $\comp{M}$-scaling cost.
We can then write
\begin{align*}
    \sum_{j=1}^{M} {\bf v}_{\tilde q+j} a\big(y^{k}_{j}, y^{k}_{i}\big) 
    =&
    \sum_{\substack{j=1 \\ V_i \cap V_j \not = \emptyset}}^{M}
    {\bf v}_{\tilde q+j} 
    a\big(y^{j},y^{i}\big) 
    +  
    a_x {\bf y}_{N_i+1}^{i} 
    +
    c_{xy}
    a_i 
    + 
    c_{xy} {\bf y}_{N_i+1}^{i}  \tilde{\bf A}_{11},
\end{align*}
to deduce from (C5) that terms of this form can be computed in $\comp{\Nloc}$-scaling cost. Thus, we conclude that the total cost of a matrix-vector product involving $\Az$ is $\Comp{M\Nloc + q(q+\Nloc)+M(q+\Nloc)}$, and therefore, using (C6), as $\comp{M\Nloc}$.

\subsubsection{The matrices \texorpdfstring{$\Si$}{Si} and \texorpdfstring{$\bold{S}^0$}{S0}}
\label{sssec:mass-matrices}
The computational cost of constructing $\Si$ and $\bold{S}^0$ or applying them to vectors scales exactly as $\Ai$ and $\Az$ respectively.

\subsubsection{The nonlinear operators \texorpdfstring{$\Ni({\bf u})$}{Ni(u)}, \texorpdfstring{$\Nz({\bf u})$}{Nz(u)}, \texorpdfstring{$\Hi({\bf u})$}{Hi(u)} and \texorpdfstring{$\Hz({\bf u})$}{Hz(u)}}
It follows from (C1) and (C4) that the computational cost of computing the term 
\[
    \tilde {\bf N}_{\ell\ell'}({\bf u}) := \int_{\Omega}G'(u^2(\bold{x})) u^{(\ell)}(\bold{x}) u^{(\ell')} (\bold{x})\; d \bold{x}
\]
for $\ell,\ell'=k-\tilde q,\ldots, k-1$ scales as $\comp{N}$, thus altogether as $\comp{q^2 N}$.
Once $\tilde {\bf N}_{(k-1)(k-1)}({\bf u})$ has been computed, the complexity analysis of a matrix-vector multiplication ${\bf v} \mapsto \Ni({\bf u}) {\bf v}$ follows the same arguments as those for the matrix $\Ai$. Matrix-vector products involving $\Ni({\bf u})$ thus scale as $\mathcal O(N_i)$. The same considerations hold true for the assembly of and matrix-vector products with $\Nz({\bf u})$, the cost of which scales as $\comp{q^2 N}$ and $\comp{M\Nloc}$ respectively. Computations involving the matrices $\Hi({\bf u})$ and $\Hz({\bf u})$ (which may be required if one uses second-order optimization methods for the nonlinear problems) have the same cost as those for $\Ni({\bf u})$ and $\Nz({\bf u})$ respectively.

\section{Application to the model problems}
\label{sec:appl-mod-pb}

The purpose of this section is to expand on the preceding Sections \ref{sec:2.3} and \ref{sec:2.4} and explain in more detail the structure of the EMDD minimization problems and the computational cost of solving these problems in the case of the four classes of models \eqref{eq:linear_sys}-\eqref{eq:nonlinear_eig}. Throughout this section we assume the setting and notation of Sections \ref{sec:2.2}-\ref{sec:2.4}.

\subsection{Model~Problem~1: Linear Source Problems of the form \texorpdfstring{\eqref{eq:linear_sys}}{(3.1)}}

In this case, the solution $y_i^{(k)}$ to the $i^{\rm th}$ EMDD local minimization problem at iteration number $k$ can be obtained by solving the $(N_i+1)$-dimensional \emph{linear} matrix equation
\begin{align}
    \label{eq:linsys_loc_matrix}
    \Ai \yi = \fiv,
\end{align}
with $\Ai \in \R^{(N_i+1) \times (N_i+1)}$ and ${\bf{f}^{(i)}}\in \R^{N_i+1}$ defined according to Definition \ref{def:loc_matrix}.

\noindent 
Similarly, the solution to the EMDD second-level minimization problem at iteration number $k$ can  be obtained by solving the $(\tilde q+M)$-dimensional \emph{linear} matrix equation
\begin{align}
    \label{eq:linsys_2ndlevel_matrix}
    \Az {\bf{u}}^{(k)} = \fz,
\end{align}
with $\Az \in \R^{(\tilde q+M) \times(\tilde q+M)}$ and $\fz\in \R^{\tilde q+M}$ defined according to Definition \ref{def:global_matrix}. 

\bigskip

Algorithm~\ref{alg:source-problems} contains a pseudocode description of the EMDD method for this class of model problems. We deduce, in particular, that the cost of one iteration of the algorithm scales as $\comp{N+\Nkry M\Nloc}$, and we note that only one global $\comp{N}$ computation has to be performed per iteration at step (1).

\begin{algorithm}
\caption{EMDD for linear/nonlinear source problems}
\label{alg:source-problems}
\begin{algorithmic}
\STATE{\textbf{Initialization:}}\vspace{2pt}
\STATE{\textit{Required quantities:} Initial guess $u^{(0)}\in\Vapp$.}
\STATE{Compute all $f(\phi_m^{i})$ in $\comp{M\Nloc}$ and $\tilde{\bf A}_{00}=a(u^{(0)},u^{(0)})$ in $\comp{N}$.}
\medskip
\STATE{\textbf{Loop:}}\vspace{2pt}
\STATE{For each $k=1,2,\ldots$}
\medskip
\begin{enumerate}[label=(\arabic*)]
   \item[] \hspace*{-1.2\leftmargin} \textit{Required quantities:} $\tilde{\bf A}_{\ell\ell'}=a(u^{(\ell)},u^{(\ell')})$ for $\ell,\ell'\in \{k-\tilde q,\ldots,k-1\}$.
    \item complete building $\fiv$ by computing $f(u^{(k-1)})$ in $\comp{N}$,
    \item
    \begin{itemize}
    \item[\textit{(i)}]
    \textbf{linear problem:} solve the local linear problems~\eqref{eq:linsys_loc_matrix} for each $i=1,\ldots,M$ in $\comp{\Nkry M\Nloc}$$^*$  using a Krylov solver,
    \item[\textit{(ii)}] \textbf{nonlinear problem:} solve the local nonlinear problems~\eqref{eq:nonlin_sys_1stlevel} for each $i=1,\ldots,M$ with an iterative solver of cost $\comp{\Nnlin M N}$$^\dagger$,
    \end{itemize}
    \item assemble $\fz$ in $\comp{M\Nloc}$,
    \item 
    \begin{itemize}
    \item[\textit{(i)}] \textbf{linear problem:} solve the second-level optimization problem~\eqref{eq:linsys_2ndlevel_matrix} in $\comp{\Nkry M \Nloc }$$^*$  using a Krylov solver,
    \item[\textit{(ii)}] \textbf{nonlinear problem:} 
    solve the second-level optimization problem~\eqref{eq:nonlin_sys_2ndlevel} of cost {$\comp{\Nnlin M\Nloc + q^2 N \Nnlin}$}$^\dagger$ ,
    \end{itemize}
    \item 
    compute $\tilde{\bf A}_{k\ell}=a(u^{(k)},u^{(\ell)})$ for $\ell=k-\tilde q+1,\ldots,k$ and its symmetric counterpart in $\comp{qN}$. 
    In the linear case, one can reduce it to $\comp{qM}$ by using the relation $\tilde{\bf A}_{k\ell}=a(u^{(k)}, u^{(\ell)})=\big(\bold{f}^{(0)}\big)^{\intercal}\bold{u}^{(\ell)}$.
\end{enumerate}
\end{algorithmic}
$\hphantom{1}^*$$\Nkry$ is the maximum number of Krylov iterations needed to solve linear systems involving $\{\Ai\}_{i=0}^M$.
\\
$\hphantom{1}^\dagger$$\Nnlin$ is the maximum number of matrix-vector products involving $\Ai$, $\Ni({\bf u})$ or $\Hi({\bf u})$ (with varying ${\bf u}$) needed to solve Equation~\eqref{eq:nonlin_sys_1stlevel} or~\eqref{eq:nonlin_sys_2ndlevel}.
\end{algorithm}

\subsection{Example Class 2: Semilinear Source Problems of the form~\texorpdfstring{\eqref{eq:nonlinear_sys}}{(3.2)}}

In this case, the solution $y^{(k)}_i$ to the $i^{\rm th}$ EMDD local minimization problem at iteration number $k$ is obtained by solving the $N_i+1$-dimensional \emph{nonlinear} matrix equation 
\begin{align}\label{eq:nonlin_sys_1stlevel}
\Ai \yi + \Ni( \yi) \yi= \fiv,
\end{align}
with $\Ai, \Ni(\cdot) \in \R^{(N_i+1) \times (N_i+1)}$ and ${\bf{f}^{(i)}}\in \R^{N_i+1}$ defined according to Definition \ref{def:loc_matrix}. 
Similarly, the solution to the second-level EMDD minimization problem at iteration number $k$ is obtained by solving the $(\tilde q+M)$-dimensional \emph{nonlinear} matrix equation 
\begin{align}
    \label{eq:nonlin_sys_2ndlevel}
    \Az {\bf{u}}^{(k)} +\Nz( {\bf{u}}^{(k)}){\bf{u}}^{(k)} = \fz, 
\end{align}
with $\Az, \Nz(\cdot) \in \R^{(\tilde q+M) \times (\tilde q+M)}$ and $\fz\in \R^{(\tilde q+M)}$ defined according to Definition \ref{def:global_matrix}.  
\smallskip

Algorithm \ref{alg:source-problems} again provides a pseudocode description of the EMDD algorithm for this class of model problems. In this case, the overall complexity for one EMDD iteration scales as $\comp{\Nnlin (q^2+M) N}$ since each subdomain has to perform independent global $\comp{N}$ computations at step (2).

\begin{remark}[qEMDD: Reducing the complexity $\comp{MN}$]\label{rem:quadratic-approximation}
The computational bottleneck in Algorithm~\ref{alg:source-problems} is Step (2). Indeed, since $V_i^{(k)}=V_i+\operatorname{span}\{u^{(k-1)}\}$, the global function $u^{(k-1)}$ will appear with a \emph{different} weight in each local iterate in each local optimization solve. This means that the matrix $\Ni(\cdot)$ (which depends nonlinearly on the current iterate) must be reassembled at a cost of $\mathcal{O}(N)$ at each iteration of each local optimization solve, thus leading to a total scaling of  $\mathcal{O}(\Nnlin M N)$. Note that this is similar to existing DD schemes that tackle nonlinear PDEs using a DD-approach (see, e.g.~\cite{lions1988schwarz,dryja1997nonlinear,lui1999schwarz,chaouqui2022linear}). 
    
    A variant of the EMDD method, which we call quadratic EMDD (qEMDD), eliminates this bottleneck by replacing $\en(y_i)$ in (2.4) with its second-order Taylor approximation $\en_2(y_i;u^{(k-1)})$ at $u^{(k-1)}$. In the linear algebra formulation introduced above, the minimizer ${\bf{y}}^i\in \R^{N_i+1}$ of this quadratic energy functional over~$V_i^{(k)}$, can be written as
    ${\bf{y}}^i={\bf{u}}+{\bf{w}}^i$, where ${\bf{u}}=(0,\dots,0,1)^\top\in\R^{N_i+1}$ and ${\bf{w}}^i$ solves the equation
    \begin{equation}
        \label{eq:Newton-loc}
        \Big(\Ai + \Hi({\bf{u}})\Big){\bf{w}}^i
        = -\Big( \big(\Ai + \Ni({\bf{u}})\big){\bf{u}} - {\bf{f}}^i\Big).
    \end{equation}
    Equation~\eqref{eq:Newton-loc} can be seen as a single step of Newton's method applied to the local problem  \eqref{eq:EMDD-step1}, with the same initial guess $u^{(k-1)}$ for every $i$. The matrices $\Hi(\cdot)$ and $\Ni(\cdot)$ are thus evaluated only once at ${\bf{u}}=(0,\dots,0,1)^\top\in\R^{N_i+1}$ with cost $\mathcal{O}(N)$. It follows that the cost of Step (2) reduces to $\mathcal{O}(N + \Nkry M \Nloc)$ with $\Nkry$ being the number of Krylov iterations required to solve~\eqref{eq:Newton-loc}. The overall complexity of one qEMDD iteration then scales as $\mathcal{O}(\Nkry M \Nloc + q^2N\Nnlin)$. Note that such a quadratic approximation may increase the number of outer EMDD iterations, though we do not observe this in the numerical tests of Section~\ref{sec:numerics}.
\end{remark}

\subsection{Example Class 3: Linear Eigenvalue Problems of the form~\texorpdfstring{\eqref{eq:linear_eig}}{(3.3)}}
\noindent   
In this case, the solution $y_i^{(k)}$ to the $i^{\rm th}$ EMDD local minimization problem at iteration number $k$ can be obtained by solving (for the lowest eigenpair) the $(N_i+1)$-dimensional \emph{generalized, linear} matrix eigenvalue problem
\begin{align}\label{eq:lineig_loc_matrix}
\Ai \yi = \lambda^{(i)} \Si\yi,
\end{align}
with $\Ai, \Si \in \R^{(N_i+1) \times (N_i+1)}$ and ${\bf{f}^{(i)}}\in \R^{N_i+1}$ defined according to Definition \ref{def:loc_matrix}.

\noindent Similarly, the solution to the second-level EMDD minimization problem at iteration number $k$ can be obtained by solving (for the lowest eigenpair) the $N^{(0)}_k$-dimensional \emph{generalized, linear} matrix equation
\begin{align}\label{eq:lineig_2ndlevel_matrix}
\Az {\bf{u}}^{(k)} = \lambda^{(k)}\Sz{\bf{u}}^{(k)},
\end{align}
with $\Az, \Sz \in \R^{(\tilde q+M) \times (\tilde q+M)}$ and $\fz\in \R^{\tilde q+M}$ defined according to Definition \ref{def:global_matrix}.

\vspace{2mm}
Let us remark here that--thanks to the min-max principle--if the lowest eigenvalue of the global eigenvalue problem \eqref{eq:linear_eig} is non-degenerate, and if the initialization $u^{(0)}$ is sufficiently accurate, then  both eigenvalue problems \eqref{eq:lineig_loc_matrix} and \eqref{eq:lineig_2ndlevel_matrix} will possess spectral gaps at their respective lowest eigenvalues.

\medskip
Algorithm~\ref{alg:evp-problems} contains a pseudocode description of the EMMD method for eigenvalue problems and implies, in particular, that the overall complexity for one iteration scales as $\comp{N + \Nevp M \Nloc}$.

\begin{algorithm}
\caption{EMDD for linear/nonlinear eigenvalue problems}
\label{alg:evp-problems}
\begin{algorithmic}
\STATE{\textbf{Initialization:}}\vspace{2pt}
\STATE{\textit{Required quantities:} Initial guess $u^{(0)}\in\Vapp$.}
\STATE{Compute $\tilde{\bf A}_{00}=a(u^{(0)},u^{(0)})$ and $\tilde{\bf S}_{00}=(u^{(0)},u^{(0)})_{L^2(\Omega)}$ in $\comp{N}$.}
\medskip
\STATE{\textbf{Loop:}}\vspace{2pt}
\STATE{For each $k=1,2,\ldots$}
\medskip
\begin{enumerate}[label=(\arabic*)]
    \item[] \hspace*{-1.2\leftmargin} 
    \textit{Required quantities:} 
    $\tilde{\bf A}_{\ell\ell'}=a(u^{(\ell)},u^{(\ell')})$ and $\tilde{\bf S}_{\ell\ell'}=(u^{(\ell)},u^{(\ell')})_{L^2(\Omega)}$ for $\ell,\ell'\in \{k-\tilde q,\ldots,k-1\}$.
    \item 
    \begin{itemize}
    \item[\textit{(i)}]
    \textbf{linear problem:} 
    solve the generalized eigenvalue problem~\eqref{eq:lineig_loc_matrix} for each $i=1,\ldots,M$ in $\comp{\Nevp M\Nloc}$$^*$ using an iterative eigenvalue solver,
    \item[\textit{(ii)}] 
    \textbf{nonlinear problem:} 
    solve the nonlinear generalized eigenvalue problem~\eqref{eq:nonlin_matrix} for each $i=1,\ldots,M$ in $\comp{\Nscf \Nevp M\Nloc + \Nscf MN }$$^\dagger$ using an SCF iterative scheme,
    \end{itemize}
    \item 
    \begin{itemize}
    \item[\textit{(i)}] \textbf{linear problem:}
    solve the second-level optimization problem~\eqref{eq:lineig_2ndlevel_matrix} in $\comp{\Nevp M \Nloc}$$^*$ using an iterative eigenvalue solver, 
    \item[\textit{(ii)}] \textbf{nonlinear problem:} 
    solve the second-level optimization problem~\eqref{eq:nonlin_matrix_2} in {$\comp{\Nscf\Nevp M\Nloc + q^2\Nscf N}$}$^\dagger$ using an SCF-iterative scheme.
    \end{itemize}
    \item compute $\tilde{\bf A}_{k\ell}=a(u^{(k)},u^{(\ell)})$ and $\tilde{\bf S}_{k\ell}=(u^{(k)},u^{(\ell)})_{L^2(\Omega)}$ for $\ell=k-\tilde q+1,\ldots,k$ and its symmetric counterpart in $\comp{qN}$. In the linear case, one can use the relation $\tilde{\bf A}_{k\ell}=\lambda^{(k)} \tilde{\bf S}_{k\ell}$.
\end{enumerate}
\end{algorithmic}
 $\hphantom{1}^*$$\Nevp$ is the maximum number of matrix-vector products involving $\{\Ai\}_{i=0}^M$, $\{\Si\}_{i=0}^M$ needed in the iterative method to solve the eigenvalue problem~\eqref{eq:lineig_loc_matrix} or~\eqref{eq:lineig_2ndlevel_matrix}.

$\hphantom{1}^\dagger$$\Nscf$ is the maximum number of outer SCF iterations required to solve the local nonlinear eigenvalue problems~\eqref{eq:nonlin_matrix} or~\eqref{eq:nonlin_matrix_2} while $\Nevp$ is the maximum number of matrix-vector products involving $\{\Ai\}_{i=0}^M$, $\{\Si\}_{i=0}^M$ used by the inner (generalized, linear) eigenvalue solver.
\end{algorithm}

\subsection{Example Class 4: Nonlinear Eigenvalue Problems of the form~\texorpdfstring{\eqref{eq:nonlinear_eig}}{(3.4)}} 
In this case, the solution $y_i^{(k)}$ to the $i^{\rm th}$ EMDD local minimization problem at iteration number $k$ can thus be obtained by solving the $(N_i+1)$-dimensional \emph{constrained} minimization problem
\begin{align}\label{eq:nonlin_energy_matrix}
    \yi
    :=& 
    \argmin_{{\bf y} \in \R^{N_i+1}}\hspace{1mm}  {
    \tfrac12 \bf y}^{\intercal}\Ai {\bf y} 
    +  \Gi({\bf{y}}) 
    \qquad \text{subject to} \quad 
    \big(\yi\big)^{\intercal} \Si\yi=1,
\end{align}
with $\Ai, \Si  \in \R^{(N_i+1) \times (N_i+1)}$ defined according to Definition \ref{def:loc_matrix} and where, $\forall y\in \Vik$ represented by the vector ${\bf y}\in \R^{N_i+1}$, $\Gi({\bf y})$ whose gradient is $\Ni({\bf y}){\bf y}$ is given by 
\[
    \Gi({\bf y})
    =
    \frac{1}{2}
     \int_{\Omega}G_{\rm eig}(y^2(\bold{x})) \; d \bold{x}.
\]

Note that the normalization constraint appearing in Equation \eqref{eq:nonlin_energy_matrix} defines a smooth embedded submanifold of $\R^{N_i+1}$, which makes Equation \eqref{eq:nonlin_energy_matrix} amenable to numerical resolution using Riemannian optimization algorithms. Alternatively, the Euler-Lagrange equations associated with the minimization problem \eqref{eq:nonlin_energy_matrix} consist of the following $(N_i+1)$-dimensional \emph{generalized, nonlinear} matrix eigenvalue problem 
\begin{align}
    \label{eq:nonlin_matrix}
    \Ai \yi + \Ni( \yi) \yi =  \lambda^{i} \Si\yi. 
\end{align}
with $\Ni(\cdot) \in \R^{(N_i+1) \times (N_i+1)}$ defined according to Definition \ref{def:loc_matrix}. Equation~\eqref{eq:nonlin_matrix} can be solved very efficiently using the so-called self-consistent field method (SCF), which consists of freezing the nonlinearity at the current eigenvector iterate, solving the resulting linear eigenvalue problem, updating the nonlinearity and repeating until convergence (see, e.g.,~\cite{cances2021convergence}). This process can be accelerated using various techniques related to Anderson acceleration \cite{anderson1965iterative}.

The solution to the second-level EMDD minimization problem at iteration number~$k$  can now be obtained by solving the following $(\tilde q +M)$-dimensional \emph{constrained} minimization problem:
\begin{align}\label{eq:nonlin_energy_matrix_2}
    {\bf{u}}^{(k)}
    :=& \argmin_{{\bf u} \in \R^{N^{(0)}_k}}\hspace{1mm}  \tfrac12{\bf u}^{\intercal}\Az {\bf u} +  \Gz({\bf{u}}) 
    \qquad \text{subject to} \quad 
    \big({\bf{u}}^{(k)}\big)^{\intercal} \Sz{\bf{u}}^{(k)}=1,
\end{align}
with $\Az, \Sz  \in \R^{(\tilde q +M) \times (\tilde q +M)}$~defined according to Definition \ref{def:global_matrix} and where, $\forall y\in \Vzk$ represented by the vector ${\bf y}\in \R^{\tilde q +M}$, $\Gz({\bf y})$ whose gradient is $\Nz({\bf y}){\bf y}$ is given by 
\[
    \Gz({\bf y})
    =
    \frac{1}{2}
     \int_{\Omega}G_{\rm eig}(y^2(\bold{x})) \; d \bold{x}.
\]

The Euler-Lagrange equations associated with the constrained minimization problem \eqref{eq:nonlin_energy_matrix_2} consist of the following $(\tilde q +M)$-dimensional \emph{generalized, nonlinear} matrix eigenvalue problem 
\begin{align}
    \label{eq:nonlin_matrix_2}
    \Az {\bf{u}}^{(k)} + \Nz( {\bf{u}}^{(k)}) {\bf{u}}^{(k)} =  \lambda^{(k)} \Sz{\bf{u}}^{(k)},
\end{align}
with $\Nz(\cdot) \in \R^{(\tilde q +M) \times (\tilde q +M)}$ defined according to Definition \ref{def:global_matrix}. As before, one can either solve the constrained minimization problem \eqref{eq:nonlin_energy_matrix_2} using Riemannian optimization methods, or the nonlinear eigenvalue problem \eqref{eq:nonlin_energy_matrix_2} using SCF-type methods. Let us note here that while both constrained minimization problems \eqref{eq:nonlin_energy_matrix} and \eqref{eq:nonlin_energy_matrix_2} admit minimizers, it is not clear these minimizers are unique (even accounting for sign). It is also not clear (unless further assumptions are imposed) that the minimizers correspond to the \emph{ground state} (i.e., lowest) eigenvalue of the nonlinear eigenvalue problems which is a known particularity for nonlinear eigenvalue problems. 

Algorithm~\ref{alg:evp-problems} contains a pseudocode description of the EMDD algorithm for nonlinear eigenvalue problems. The overall complexity of the algorithm per iteration scales as $\comp{\Nscf\Nevp M\Nloc + q^2\Nscf N}$. Here the first term denotes the cost of performing matrix vector products within each SCF outer loop while the second term denotes the cost of updating the nonlinearity (once per outer SCF iteration).

\begin{remark}[qEMDD: Reducing the complexity $\comp{MN}$]
    \label{rem:quadratic-approximation-nevp}
        As was the case for the nonlinear source problem \eqref{eq:nonlinear_sys} (see Remark~\ref{rem:quadratic-approximation}), we can reduce the computational cost of Step (1) of Algorithm~\ref{alg:evp-problems} for nonlinear eigenvalue problems by introducing a quadratic approximation of the energy functional.
    Indeed, denoting by $\mathcal R(v)=v/\|v\|_{L^2(\Omega)}$, the retraction on the unit sphere, and expanding  $w \mapsto \en(\mathcal R(u^{(k-1)}+w))$ around $w=0$, defines the second order Taylor approximation of $\en\circ\mathcal R$ around $u^{(k-1)}$.
    Minimizing the resulting quadratic energy functional over $w^i\in V_i^{(k)}\cap\mathcal T_{u^{(k-1)}}\mathcal M$ then amounts to solving
    \begin{equation*}
        \Pu^\top \Big(\Ai + \Hi({\bf{u}}) - \lambda \Si \Big) \Pu\, {\bf{w}}^i
        = -\Big( \Ai + \Ni({\bf{u}}) - \lambda \Si\Big){\bf{u}},
    \end{equation*}
    where ${\bf{u}}=(0,\dots,0,1)^\top\in\R^{N_i+1}$ collects the coordinates of $u^{(k-1)}$ in the basis of $V_i^{(k)}=V_i+\operatorname{span}\{u^{(k-1)}\}$, so that ${\bf{u}}^\top\Si{\bf{u}}=1$ and $\lambda={\bf{u}}^\top(\Ai+\Ni({\bf{u}})){\bf{u}}$, and where $\Pu = {\bf{I}} - {\bf{u}} {\bf{u}}^\top \Si$ is the projection onto the tangent space $\{ {\bf{z}}\in \R^{N_i+1} \;|\; {\bf{z}}^\top \Si {\bf{u}}=0 \}$; the matrix inversion is to be understood in the tangent space only.
    The corrections $w^i$ then replace $y^{(k)}_i$ in Step (2) of~\eqref{eq:EMDD-step1}.
\end{remark}

\section{Numerical Experiments} \label{sec:numerics}
We present numerical experiments for the EMDD method applied to the four example classes introduced in Section \ref{sec:2.2}. The experiments are performed using a serial implementation of the EMDD method using the \texttt{Gridap} finite element framework~\cite{badiaGridapExtensibleFinite2020} in the Julia language~\cite{bezansonJuliaFreshApproach2017} and the code is archived in~\cite{codeZenodo2026}. In the following, we always denote the maximal subdomain diameter by $H$, the number of overlap layers, i.e.\ the amount of element added to each subdomain, by $\ell$, the maximal physical overlap size by $\delta$, and the mesh size by $h$. If not otherwise stated, we use an irregular METIS~\cite{karypisMETISSoftwarePackage1997} partition to construct the subdomain decomposition.

\subsection{Linear Source Problems}\label{sec:numerics_linear_sys}
We first consider linear source problems
\begin{align*}
    -\nabla \cdot (\alpha(x,y)\nabla u) = f \quad \text{in } \Omega, \qquad u = 0 \quad \text{on } \partial \Omega,
\end{align*}
where $\Omega = (0,1)^2$ is the unit square and $f$ is a forcing term. Two different problems are considered: the Poisson problem with $\alpha = 1, f = 1$ and a diffusion problem with a spatially varying coefficient $\alpha(x,y) = 2 + \sin(2\pi x + 3\pi y)$ together with a forcing term $f(x,y)$, that is manufactured to yield the asymmetric and sign-changing solution $u(x,y) = e^{3 x y}\sin(\pi x)\sin(2\pi y)$.
\par
We discretize the problem using a first-order conforming finite-elements on a uniform quadrilateral mesh with $1/h = 64$. The resulting linear system is solved using the EMDD method with $q=1$ and $q=2$, as well as with the restricted additive Schwarz (RAS) iteration, the conjugate gradient method with symmetric additive-Schwarz preconditioning (CG+AS), and right-preconditioned GMRES with the non-symmetric RAS preconditioning (GMRES+RAS). All methods start from the all-ones coefficient vector, \(\uz = \boldsymbol{1}\). With the residual defined as \(\resk \coloneq \bA\uk - \bfb\), the convergence of the methods is measured relatively to the initial residual norm, i.e., \(\|\resk\|_2 / \|\resz\|_2\). 
Here, $\bA$ denotes the matrix representation of the linear differential operator in the global finite element basis.
\par
In Fig.~\ref{fig:poisson_cmp}, we present the convergence of the methods for two overlap layers (\(\ell=2\)) and $m \in \{4,16,64\}$ subdomains for both problem types and plot the relative residual norm for each outer iteration, i.e., a complete local subdomain solution procedure for the EMDD method. We observe that EMDD with \(q=2\) is consistently much faster than \(q=1\). All methods require more iterations to reach a fixed accuracy as \(m\) grows, as expected without a coarse correction while the history vector substantially reduces this deterioration but does not remove it. For the nontrivial diffusion problem, we observe an increase of iterations for the EMDD method, whereas the other methods do not behave different compared to the standard Poisson problem.
\begin{figure}[t]%
  \includegraphics[width=1.0\textwidth]{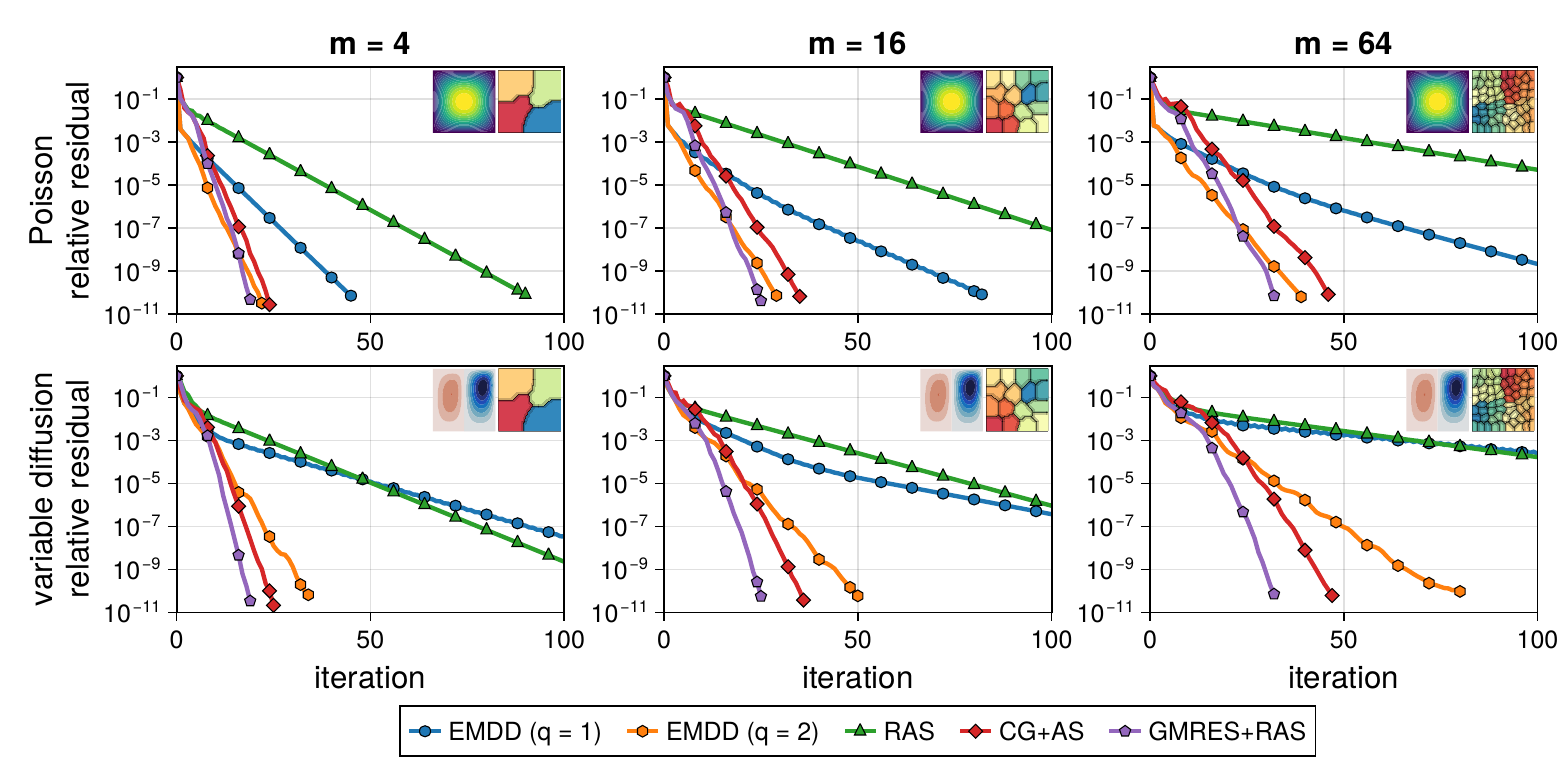}%
  \caption{Convergence of EMDD ($q=1$ and $q=2$), stationary RAS, CG+AS, and GMRES+RAS with $1/h=64$, two overlap layers, and $m=4,16,64$ subdomains for the Poisson problem (first row) and the variable diffusion problem (second row). The insets show the corresponding element partitions, in which darker interface regions indicate the overlap multiplicity.}\label{fig:poisson_cmp}%
\end{figure}%
\providecommand{\varddtablepath}{./tables}%
\providecommand{\varddtablepath}{./examples/paper/tables}
\begin{table*}[t]
  \footnotesize
  \centering
  \caption{%
    Parallel local-solve iterations required to reduce the relative residual below \(10^{-10}\) for the Poisson problem. The first two column groups refine the mesh for \(m=4\) while keeping the overlap layers fixed (first group) and keeping the physical overlap size fixed (second group). The third varies the overlap for \(1/h=64\) and \(m=4\).
  }\label{tab:poisson-scaling}
  \resizebox{\textwidth}{!}{%
  \pgfplotstabletypeset[
    col sep=comma,
    columns={method,fixed_layers_h20,fixed_layers_h40,fixed_layers_h60,
      fixed_layers_h80,fixed_layers_h100,fixed_layers_h120,
      fixed_ratio_h20,fixed_ratio_h40,fixed_ratio_h60,fixed_ratio_h80,
      fixed_ratio_h100,fixed_ratio_h120,overlap1,overlap2,overlap4,overlap8},
    every head row/.style={
      output empty row,
      before row={
        \toprule
        & \multicolumn{6}{c}{\makebox[0pt]{discretization \(1/h\) (\(\ell=2\))}}
        & \multicolumn{6}{c}{\makebox[0pt]{discretization \(1/h\) (\(\delta/H\approx0.1\))}}
        & \multicolumn{4}{c}{\makebox[0pt]{layers \(\ell\) (\(1/h=64\))}} \\
        \cmidrule(lr){2-7} \cmidrule(lr){8-13} \cmidrule(lr){14-17}
        method & 20 & 40 & 60 & 80 & 100 & 120
        & 20 & 40 & 60 & 80 & 100 & 120
        & 1 & 2 & 4 & 8 \\
      },
      after row=\midrule
    },
    every last row/.style={after row=\bottomrule},
    every column/.style={string type,column type={r}},
    columns/method/.style={string type,column type={l@{\hspace{1em}}}},
  ]{\varddtablepath/fig18_poisson_scaling.csv}%
  }
\end{table*}

\par
The sensitivity of the EMDD method with respect to the mesh and overlap properties is further presented in Table~\ref{tab:poisson-scaling} for a Cartesian partition. The first mesh sequence keeps two overlap layers fixed for an increasing mesh resolution. As expected, iteration counts generally grow as \(h\) decreases since the physical overlap size shrinks. The second sequence, on the other hand, increases the layer count with refinement so that \(\delta/H\), the ratio of the physical overlap size ($\approx \ell h$) to the subdomain size $H$, remains approximately \(0.1\). In this case, the iteration counts nearly level off, illustrating that preserving physical overlap recovers much better mesh-independent behavior. The final block fixes \(1/h=64\) and varies the overlap from one to eight layers. Increasing overlap generally lowers the iteration count. These results match with the classical theory of one-level domain decomposition methods~\cite{toselliDomainDecompositionMethods2005}. We additionally observe that increasing the EMDD history size above $q=2$ does not drastically improve the convergence, a behavior that we find for all test cases considered in this work.

\subsection{Semilinear Source Problems}\label{sec:numerics_semilinear_sys}
We continue with a strictly convex energy \(\mathcal{E}(u)=\frac12\int_\Omega |\nabla u|^2 +\frac{\beta}{4}\int_\Omega u^4-\int_\Omega fu\) on the unit square $\Omega=(0,1)^2$ for some $\beta>0$. The corresponding Euler-Lagrange equation reads
\begin{equation}\label{eq:numerics_semilinear_poisson}
  -\Delta u + \beta u^3 = f \quad \text{in } \Omega,
  \qquad u=0 \quad \text{on } \partial\Omega.
\end{equation}%
The forcing term is manufactured from
\begin{equation}\label{eq:numerics_semilinear_exact}
  u(x,y)
  =1.5\sin(\pi x)\sin(\pi y)
  +0.55\sin(2\pi x)\sin(3\pi y)
  +0.35\sin(3\pi x)\sin(2\pi y),
\end{equation}
and hence $f=-\Delta u+\beta u^3$. This choice gives a non-trivial, asymmetric and smooth solution while allowing the nonlinearity to be varied without changing the exact solution. We use conforming triangular $P_1$ finite elements on a uniform mesh with $1/h=16$ and partitions with $m\in \{2,4,8\}$ subdomains, and two overlap layers. All methods start from the zero function, which disables the nonlinearity in the first step. Let $\resk=\nabla \mathcal{E}_h(\uk)$ denote the assembled discrete Euler--Lagrange residual, we then report $\|\resk\|_2/\|\resz\|_2$ and stop when this quantity is below $10^{-7}$.
\begin{figure}[t]%
  \includegraphics[width=1.0\textwidth]{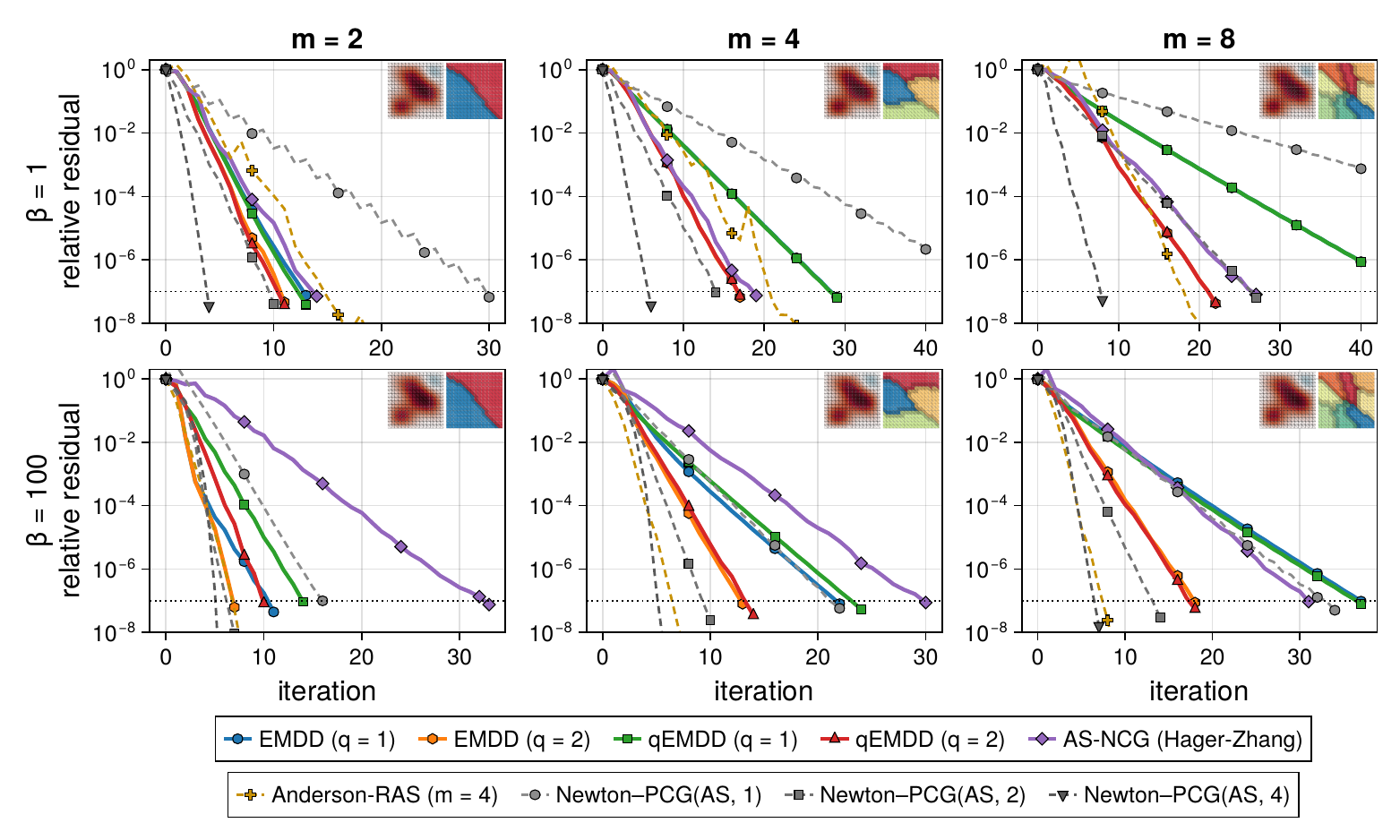}%
  \caption{%
  Convergence for the semilinear problem
  \eqref{eq:numerics_semilinear_poisson} with $\beta=1$ (top) and $\beta=100$
  (bottom), $1/h=16$, two overlap layers, and $m=2,4,8$ subdomains.%
  }\label{fig:semilinear_cmp}%
\end{figure}%
\par
Fig.~\ref{fig:semilinear_cmp} compares EMDD with $q=1$ and $q=2$ for $\beta=1$ and $\beta=100$. We additionally consider the quadratic-model EMDD (qEMDD) of Remark~\ref{rem:quadratic-approximation}, again with $q=1,2$. At the beginning of each qEMDD iteration, the energy is replaced by a quadratic Taylor approximation, which is then optimized following Remark~\ref{rem:quadratic-approximation}. The second-level recombination still minimizes the full nonlinear energy. The main comparison method is an additive-Schwarz preconditioned nonlinear conjugate gradient method with the Hager--Zhang update from the \texttt{Optim.jl}\footnote{\url{https://github.com/JuliaNLSolvers/Optim.jl}} package. We also report an Anderson-accelerated nonlinear RAS with history depth of four from the \texttt{NonlinearSolve.jl}\footnote{\url{https://github.com/SciML/NonlinearSolve.jl}} package and inexact Newton--PCG with one-level additive-Schwarz preconditioning and $\nu=1,2,4$ inner PCG steps.
\par
As before, the EMDD method with \(q=2\) is consistently faster than \(q=1\). Interestingly, the quadratic local model demonstrates the same effective convergence behavior as its exact nonlinear variant but at a fraction of the cost. 
With regard to the Newton--PCG method, we observe the usual second-order convergence behavior -- provided that a sufficient number of inner PCG iterations are carried out. Note however, that if the inverse Hessian is not computed to sufficient accuracy (e.g., if only one PCG iteration is performed), then the convergence of the Newton method is very slow. The Anderson-accelerated nonlinear RAS method works for the easier case of $\beta=100$ but has monotonicity difficulties for $\beta=1$.
Note that a careful evaluation of the efficiency considers not only the number of outer iterations but can be measured by the number a DD preconditioner is applied, making the (quasi-)Newton methods less efficient than they appear in Fig.~\ref{fig:semilinear_cmp}.

As an example of a well-posed semi-linear problem that does not satisfy the sufficient conditions \eqref{eq:irregular},
we consider the diffusion--reaction benchmark of~\cite[Sec.~6.1]{spicher2026iterative}:%
\begin{equation}
    \label{eq:numerics_semilinear_lshape}
    -\Delta u-12\exp(-u^2)=g \quad\text{in }\Omega,
  \qquad u=0\quad\text{on }\partial\Omega,
\end{equation}
in the L-shaped domain $\Omega=(-1,1)^2\setminus([-1,0]\times[0,1])$. The forcing is chosen so~that
\begin{equation}
  u(x,y)=2xy(1-x^2)(1-y^2)(x^2+y^2)^{-2/3}
\end{equation} 
is the exact solution. We employ a graded triangular $P_1$ mesh, two overlap layers, and the same values $m\in\{2,4,8\}$. The initial iterate, residual criterion, and comparison methods are the same as in the previous example. 
Fig.~\ref{fig:semilinear_lshape} confirms the benefit of the history vector for this problem with irregular solution. We also note that the more efficient quadratic-model variant is again almost equivalent to the exact variant.
\begin{figure}[t]%
  \includegraphics[width=1.0\textwidth]{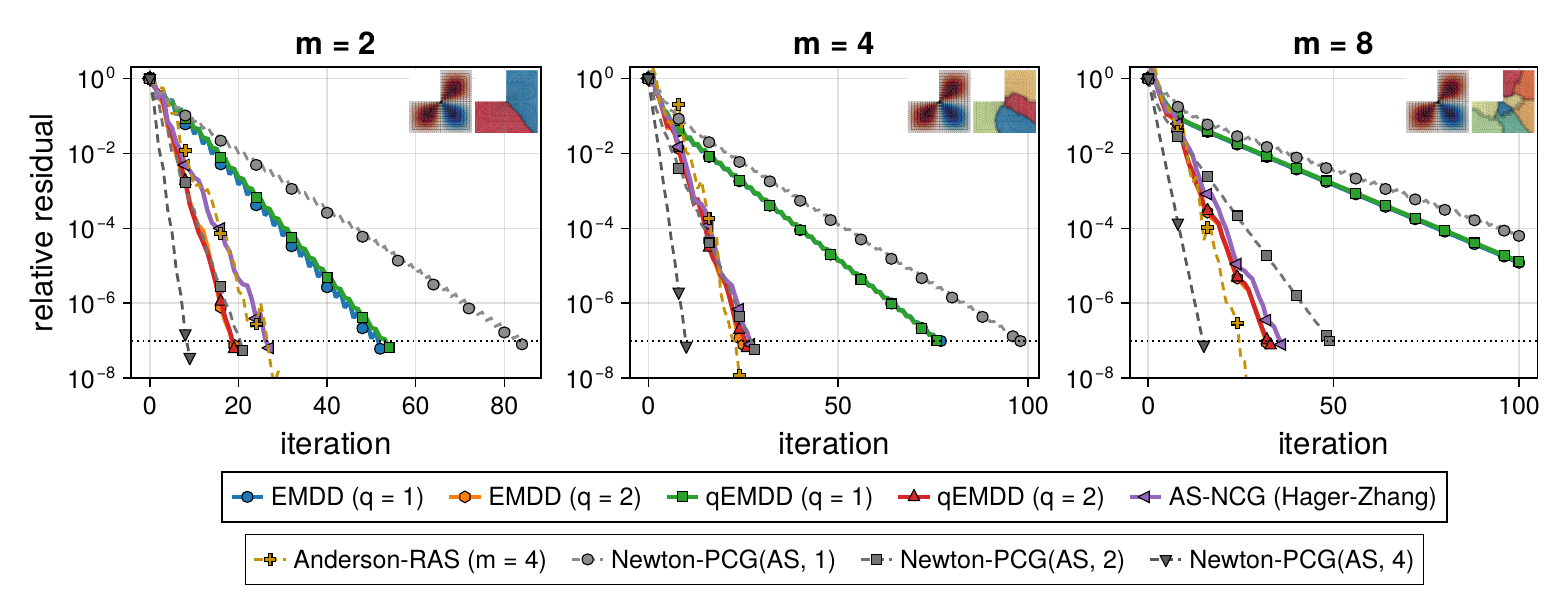}%
  \caption{%
    Convergence for the semilinear problem \eqref{eq:numerics_semilinear_lshape}  on the graded L-shaped domain, with $N=16$, grading exponent $0.4$, two overlap layers, and $m=2,4,8$ subdomains.
  }\label{fig:semilinear_lshape}%
\end{figure}%

\subsection{Linear Eigenvalue Problems}

We next consider the computation of the smallest eigenpair of the linear Schr\"odinger-type eigenvalue problem
\begin{equation}\label{eq:evp_experiment}
  -\Delta u + V(x,y)\,u=\lambda u
  \quad\text{in }\Omega,
  \qquad u=0\quad\text{on }\partial\Omega,
\end{equation}
on the unit square $\Omega=(0,1)^2$, with the off-centered exponential potential \(V(x,y)=\exp (5\sqrt{2(x-0.25)^2+(y-0.70)^2})\). In this case, we employ conforming bilinear $Q_1$ finite elements on a uniform quadrilateral mesh with $1/h=64$ and two overlap layers with $m\in \{4,16,64\}$ subdomains. Fig.~\ref{fig:evp_cmp} reports the convergence behavior in comparison with the classical gradient-based locally optimal preconditioned steepest descent (LOPSD+AS) and locally optimal block preconditioned conjugate gradient (LOBPCG+AS) methods (both with one-level additive Schwarz preconditioners).
For the sake of completeness, we also consider the Jacobi--Davidson method, see, e.g.,~\cite{sleijpen1996jacobi}, with a Schwarz-preconditioned GMRES solver, where $\nu=1,2,4$ inner GMRES iterations are performed per outer step, denoted by JD-GMRES(AS,~$\nu$). All methods start from the same $L^2$-normalized constant initial vector and, with $\resk\coloneq\bf{K}\uk-\lambda_k\bS\uk$ the assembled discrete residual, where $\lambda_k=(\uk)^\top\bA\uk$ is Rayleigh quotient for the $L^2$-normalized iterate $\uk$, we report $\Vert\resk\Vert_2/\Vert\resz\Vert_2$ and stop when this quantity is lower $10^{-6}$. In Fig.~\ref{fig:evp_cmp}, we see that the EMDD method with \(q=1\) is similar to LOPSD+AS, while the EMDD with \(q=2\) is comparable to LOBPCG+AS, which matches the same dimensions of the local subspaces to choose the next iterate from. As expected, the JD-GMRES(AS,~$\nu$) methods, which involve an approximate linear system inversion and are therefore considerably more expensive per outer iteration, outperform the other methods for $\nu=2,4$.
\begin{figure}[t]%
  \includegraphics[width=1.0\textwidth]{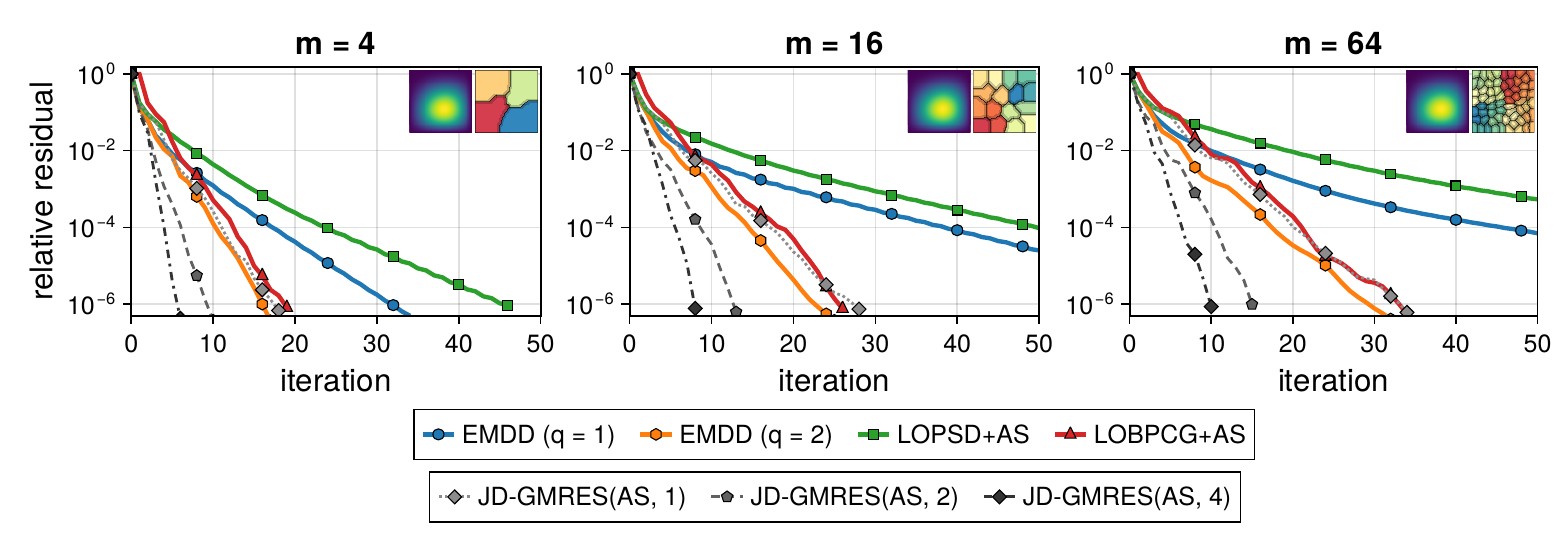}%
  \caption{%
    Convergence for the linear eigenvalue problem \eqref{eq:evp_experiment} obtained with EMDD ($q=1$ and $q=2$), LOPSD+AS, LOBPCG+AS, and JD-GMRES(AS,~$\nu$) with $\nu=1,2,4$ inner GMRES iterations.%
  }%
  \label{fig:evp_cmp}%
\end{figure}%

\subsection{Nonlinear Eigenvalue Problems}
As an example of a nonlinear eigenvalue problem, we compute the ground state of the Gross--Pitaevskii eigenvalue problem
\begin{equation}\label{eq:gp_experiment}
  -\Delta u + Vu + \beta u^3 = \lambda u \quad\text{in }\Omega,
  \qquad u = 0 \quad\text{on }\partial\Omega,
  \qquad \Vert u\Vert_{L^2(\Omega)} = 1,
\end{equation}
i.e., $a(v,w)=\int_\Omega (\nabla v\cdot\nabla w + Vvw) d \bold{x}$ and $G_{\rm eig}(t)=\tfrac{\beta}{2}t^2$ in the setting of Section \ref{sec:2.2} where \(\beta > 0\). We consider the benchmark numerical test given in \cite[Sec.~2.3]{henningGrossPitaevskiiEquation2025} which sets $\Omega=(-8,8)^2$, interaction strength $\beta=500$, and the harmonic-plus-optical trapping potential $V(x,y)=\tfrac12(x^2+4y^2)+10(\sin^2(\pi x)+\sin^2(\pi y))$. We again use conforming \(Q_1\) finite elements on a uniform $32\times32$ mesh and partitions with $m\in\{2,4,8\}$ and two overlap layers, and the prescribed $L^2$-normalized initial state $u^{(0)}(x,y)=c\,(x^2-8^2)(y^2-8^2)$ of~\cite{henningGrossPitaevskiiEquation2025}. Convergence is measured in the Euler--Lagrange residual $\resk\coloneq\big(\bA+\bN(\uk)-\lambda^{(k)}\bS\big)\uk$, where $\lambda^{(k)}$ is the current approximation of the nonlinear eigenvalue, and the tolerance is set to $10^{-6}$.
For EMDD, all resulting nonlinear eigenproblems are solved by a damped self-consistent-field iteration in $\Si$-orthonormal coordinates.
\par
In Fig.~\ref{fig:gp_cmp}, we compare EMDD with $q=1,2$ against normalized-gradient-flow methods on the same overlapping partitions. The iterate-dependent energy inner product $a_u(v,w)=\int_\Omega (\nabla v\cdot\nabla w+(V+\beta u^2)vw)\,\mathrm d\boldsymbol{x}$ yields the energy-adaptive normalized $a_u$-Sobolev gradient flow GFDN($a_u$) of~\cite{henningGrossPitaevskiiEquation2025}. We use the energy-minimizing normalized line search proposed in \cite{henningGrossPitaevskiiEquation2025} for the step size $\tau_n$ and also test the corresponding CG variant from~\cite{henningGrossPitaevskiiEquation2025}, which modifies the preceding direction with a Fletcher--Reeves-type coefficient. In principle, an exact Riemannian-gradient evaluation would require one solve with $\bA+\bN(\uk)$ per outer iteration. In our inexact realizations however, we apply $\nu \in \{1,2,4\}$ steps of additive-Schwarz-preconditioned CG as an \emph{inner solver} to the current metric equation $(\bA+\bN(\uk))\bw=\bS\uk$, warm-started with $\bw_0=\uk/\lambda^{(k)}$. These methods are denoted GFDN($a_u$)-PCG(AS,~$\nu$) and CG-GFDN($a_u$)-PCG(AS,~$\nu$).

We also include the qEMDD method ($q=1,2$), the quadratic-model variant for nonlinear eigenvalue problems introduced in Remark~\ref{rem:quadratic-approximation-nevp}. 
Since this is a constrained optimization problem on the unit sphere, qEMDD uses a second order Taylor approximation to the energy locally around $u^{(k-1)}$ on the manifold $\mathcal M$ and thus uses the Riemannian Hessian $\Pu^\top \big(\Ai + \Hi({\bf{u}}) - \lambda \Si \big) \Pu$ introduced in Remark~\eqref{rem:quadratic-approximation-nevp}.
\begin{figure}[t]%
  \includegraphics[width=1.0\textwidth]{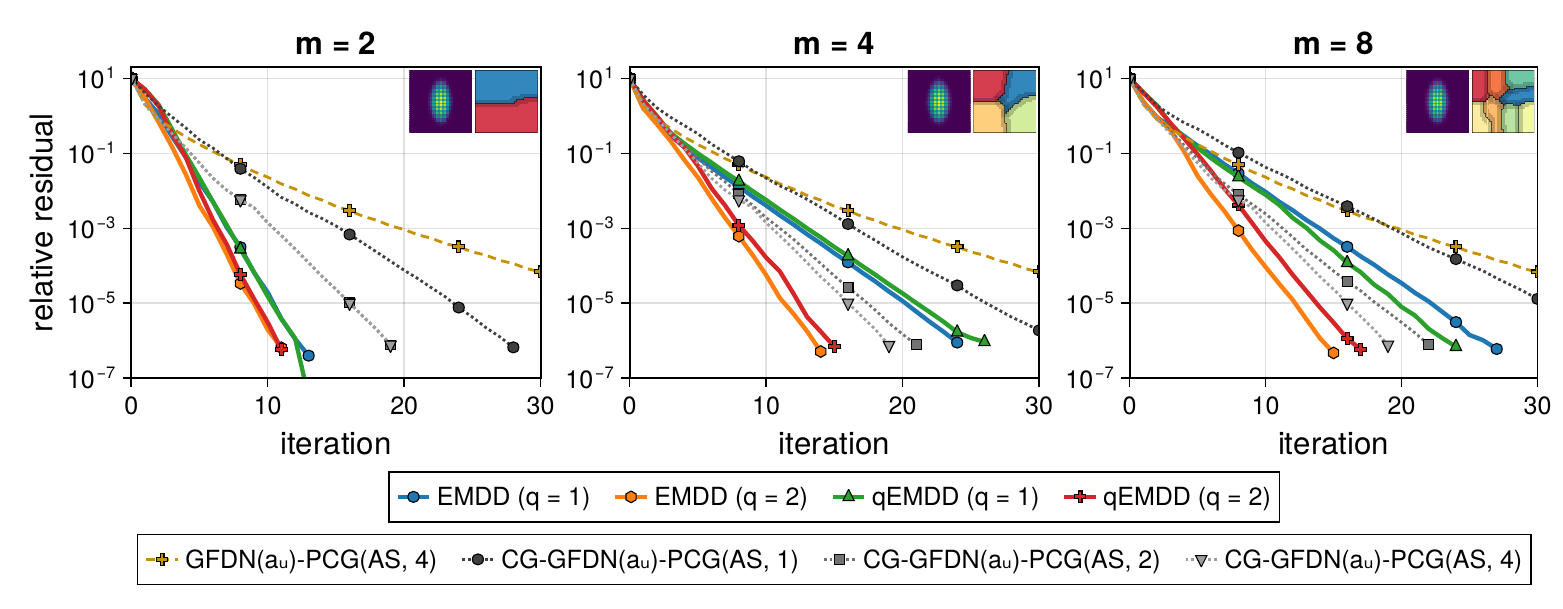}%
  \caption{%
    Convergence for the Gross--Pitaevskii ground state with $\beta=500$ obtained with EMDD and projected qEMDD ($q=1$ and $q=2$), and the energy-adaptive Sobolev gradient flows GFDN-PCG(AS,4) and CG-GFDN-PCG(AS,~$\nu$) with $\nu \in \{1,2,4\}$ inner PCG steps.%
  }%
  \label{fig:gp_cmp}%
\end{figure}%
\par
In Fig.~\ref{fig:gp_cmp}, we see that EMDD with $q=2$ is again consistently faster than $q=1$, and the projected qEMDD variant is a very good approximation of the full EMDD method but with a significantly lower cost. All methods converge, the CG variants of the GFDN method need less iterations than the plain GFDN method, and the EMDD (\(q=2\)) version needs the fewest outer iterations.

\section{Outlook}
Our purpose in this work has been to introduce a general framework for variational domain decomposition methods based on energetic principles, and to propose the additive, two-level EMDD method \eqref{eq:EMDD-step1}-\eqref{eq:EMDD-step2} as a first example of such energy-based DD methods. Numerical tests across a wide set of model problems in different regimes show that, despite its simplicity, EMDD is amongst the best-performing methods when compared to a number of model-specific DD-based methods.
We therefore believe that energy-based domain decomposition promises to be a robust numerical framework for energy-based minimization problems.

It is worth noting that more sophisticated variants of our additive, two-level EMDD method are certainly possible. We have not explored these variants in the present work, essentially because we wished to focus on the introduction of the general energy-based DD framework. Let us nevertheless, briefly mention some obvious possibilities:
\begin{enumerate}
    \item \textbf{Multiplicative EMDD}: Replace the parallel local problems \eqref{eq:EMDD-step1} by a sequential version
    \begin{equation}\label{eq:emdd-multiplicative}
      w_0^{(k)}=u^{(k-1)},\quad
      \tilde{V}_i^{(k)}=V_i+\vspan\{w_{i-1}^{(k)}\},\quad
      w_i^{(k)}=\argmin_{w\in\mathcal M\cap\tilde{V}_i^{(k)}}\en(w),
    \end{equation}
    and the vectors $w_i^{(k)}$ are used in the second level minimization \eqref{eq:EMDD-step2}. Thus, each local solve sees all preceding updates, but the local stage is no longer parallel. Such a sequential strategy was proposed in \cite{maliassov1998schwarz} and also the idea in \cite{luiRecentResultsDomain1996} is similar.
    \item \textbf{Restricted EMDD}. Use that the local minimizers can be written as $y_i^{(k)}=\alpha_i^{(k)}u^{(k-1)}+z_i^{(k)}$, with $z_i^{(k)}\in V_i$, and apply diagonal partition-of-unity operators $D_i$ (with $\sum_{i=1}^M D_i=I$) to replace $y_i^{(k)}$ in the space of \eqref{eq:EMDD-step2} by $y_i^{(k)}=D_i z_i^{(k)}$. The local problems remain those of \eqref{eq:EMDD-step1} and can still be solved in parallel.
    \item \textbf{Coarse-space enriched EMDD}. Construct a problem-adapted coarse space $W_0$ and include it in the second-level space of \eqref{eq:EMDD-step2} with \(V_{0,W}^{(k)}=\Vzk+W_0\).
  \end{enumerate}

  \vspace{1mm}
  In addition to the further development and testing of these variants and the incorporation of more sophisticated basis functions from modern FEM theory, the EMDD algorithm raises several other research questions. For instance, it would be interesting to explore the applicability of this energy-based DD methodology to more general constrained optimization problems. Obvious examples include Stiefel/Grassmann manifold optimization and nonlinear eigenvalue problems with nonlinearity in the eigenvalue. Another important question pertains to the development of a convergence analysis with explicit rates for the EMDD method, particularly in the case of eigenvalue problems. These and other questions will be the subject of future work.

\section*{Acknowledgments}
MH acknowledges support from the Deutsche Forschungsgemeinschaft (DFG) through the Emmy Noether Programme (Project No.\ 555300205). LT was supported by the German Research Foundation (DFG) through project 442047500. 

The authors acknowledge the use of generative AI for language revision and assistance with numerical implementation. All AI-assisted output was critically~reviewed and revised by the authors and they take full responsibility for the presented content.

\apptocmd{\sloppy}{\vbadness10000\relax}{}{}
\apptocmd{\sloppy}{\hbadness10000\relax}{}{}
\bibliographystyle{siamplain}
\bibliography{refs}

@article{cances2010numerical,
  title={Numerical analysis of nonlinear eigenvalue problems},
  author={Canc{\`e}s, Eric and Chakir, Rachida and Maday, Yvon},
  journal={Journal of Scientific Computing},
  volume={45},
  number={1},
  pages={90--117},
  year={2010},
  publisher={Springer}
}

@article{cances2021convergence,
  title={Convergence analysis of direct minimization and self-consistent iterations},
  author={Canc{\`e}s, Eric and Kemlin, Gaspard and Levitt, Antoine},
  journal={SIAM Journal on Matrix Analysis and Applications},
  volume={42},
  number={1},
  pages={243--274},
  year={2021},
  publisher={SIAM}
}

@techreport{karypisMETISSoftwarePackage1997,
  title = {{{METIS}}: {{A}} Software Package for Partitioning Unstructured Graphs, Partitioning Meshes, and Computing Fill-Reducing Orderings of Sparse Matrices},
  shorttitle = {{{METIS}}},
  author = {Karypis, George and Kumar, Vipin},
  year = {1997},
  pages = {1--31},
  institution = {University of Minnesota}
}

@book{toselliDomainDecompositionMethods2005,
  title = {Domain Decomposition Methods--Algorithms and Theory},
  author = {Toselli, Andrea and Widlund, Olof B.},
  year = {2005},
  series = {Springer Series in Computational Mathematics},
  number = {34},
  publisher = {Springer},
  address = {Berlin},
  xdoi = {10.1007/b137868},
  isbn = {978-3-540-20696-5},
  langid = {english},
  lccn = {QA402.2 .T67 2005}
}

@article{bezansonJuliaFreshApproach2017,
  title = {Julia: {{A Fresh Approach}} to {{Numerical Computing}}},
  shorttitle = {Julia},
  author = {Bezanson, Jeff and Edelman, Alan and Karpinski, Stefan and Shah, Viral B.},
  year = {2017},
  month = jan,
  journal = {SIAM Rev.},
  volume = {59},
  number = {1},
  pages = {65--98},
  publisher = {{Society for Industrial and Applied Mathematics}},
  issn = {0036-1445},
  xdoi = {10.1137/141000671},
  urldate = {2021-06-07}
}

@article{henningGrossPitaevskiiEquation2025,
  title = {The {{Gross--Pitaevskii}} Equation and Eigenvector Nonlinearities: {{Numerical}} Methods and Algorithms},
  shorttitle = {The {{Gross--Pitaevskii}} Equation and Eigenvector Nonlinearities},
  author = {Henning, Patrick and Jarlebring, Elias},
  year = {2025},
  journal = {SIAM Rev.},
  volume = {67},
  number = {2},
  pages = {256--317},
  publisher = {{Society for Industrial and Applied Mathematics}},
  xdoi = {10.1137/22M1516324}
}

@book{toselli2004domain,
  title={Domain decomposition methods-algorithms and theory},
  author={Toselli, Andrea and Widlund, Olof},
  volume={34},
  year={2004},
  publisher={Springer Science \& Business Media}
}

@book{quarteroni1999domain,
  title={Domain decomposition methods for partial differential equations},
  author={Quarteroni, Alfio and Valli, Alberto},
  year={1999},
  publisher={Oxford University Press}
}

@misc{spicher2026iterative,
  author={Spicher, Florian and Wihler, Thomas P.},
  title={Iterative Finite Element Approximation of Non-Monotone Semilinear Diffusion-Reaction Equations},
  year={2026},
  xeprint={2607.07327},
  xarchivePrefix={arXiv},
  xprimaryClass={math.NA},
  note={Preprint, arXiv:2607.07327},
  xdoi={10.48550/arXiv.2607.07327}
}

@book{Smith1996Domain,
  author    = {Smith, Barry F. and Bj{\o}rstad, Petter E. and Gropp, William D.},
  title     = {{Domain Decomposition: Parallel Multilevel Methods for Elliptic Partial Differential Equations}},
  publisher = {Cambridge University Press},
  address   = {Cambridge, UK},
  year      = {1996},
  isbn      = {978-0-521-49589-9}
}

@article{cai2002nonlinearly,
  title={{Nonlinearly preconditioned inexact Newton algorithms}},
  author={Cai, Xiao-Chuan and Keyes, David E},
  journal={SIAM Journal on Scientific Computing},
  volume={24},
  number={1},
  pages={183--200},
  year={2002},
  publisher={SIAM}
}

@article{cai2002non,
  title={{Non-linear additive Schwarz preconditioners and application in computational fluid dynamics}},
  author={Cai, Xiao-Chuan and Keyes, David E and Marcinkowski, Leszek},
  journal={International journal for numerical methods in fluids},
  volume={40},
  number={12},
  pages={1463--1470},
  year={2002},
  publisher={Wiley Online Library}
}

@article{hwang2007class,
  title={{A class of parallel two-level nonlinear Schwarz preconditioned inexact Newton algorithms}},
  author={Hwang, Feng-Nan and Cai, Xiao-Chuan},
  journal={Computer methods in applied mechanics and engineering},
  volume={196},
  number={8},
  pages={1603--1611},
  year={2007},
  publisher={Elsevier}
}

@article{liu2015field,
  title={{Field-split preconditioned inexact Newton algorithms}},
  author={Liu, Lulu and Keyes, David E},
  journal={SIAM Journal on Scientific Computing},
  volume={37},
  number={3},
  pages={A1388--A1409},
  year={2015},
  publisher={SIAM}
}

@article{dolean2016nonlinear,
  title={{Nonlinear preconditioning: How to use a nonlinear Schwarz method to precondition Newton's method}},
  author={Dolean, Victorita and Gander, Martin J and Kheriji, Walid and Kwok, Felix and Masson, Roland},
  journal={SIAM Journal on Scientific Computing},
  volume={38},
  number={6},
  pages={A3357--A3380},
  year={2016},
  publisher={SIAM}
}

@article{klawonn2014nonlinear,
  title={Nonlinear feti-dp and bddc methods},
  author={Klawonn, Axel and Lanser, Martin and Rheinbach, Oliver},
  journal={SIAM Journal on Scientific Computing},
  volume={36},
  number={2},
  pages={A737--A765},
  year={2014},
  publisher={SIAM}
}

@article{klawonn2017nonlinear,
  title={{Nonlinear FETI-DP and BDDC methods: a unified framework and parallel results}},
  author={Klawonn, Axel and Lanser, Martin and Rheinbach, Oliver and Uran, Matthias},
  journal={SIAM Journal on Scientific Computing},
  volume={39},
  number={6},
  pages={C417--C451},
  year={2017},
  publisher={SIAM}
}

@article{cai1994domain,
  title={{Domain decomposition methods for monotone nonlinear elliptic problems}},
  author={Cai, Xiao-Chuan and Dryja, Maksymilian},
  journal={Contemporary mathematics},
  volume={180},
  pages={21--21},
  year={1994},
  publisher={American Mathematical Society}
}

@article{cai1998parallel,
  title={{Parallel Newton--Krylov--Schwarz algorithms for the transonic full potential equation}},
  author={Cai, Xiao-Chuan and Gropp, William D and Keyes, David E and Melvin, Robin G and Young, David P},
  journal={SIAM Journal on Scientific Computing},
  volume={19},
  number={1},
  pages={246--265},
  year={1998},
  publisher={SIAM}
}

@inproceedings{cai1994newton,
  title={{Newton-Krylov-Schwarz methods in CFD}},
  author={Cai, X-C and Gropp, William D and Keyes, David E and Tidriri, Moulay D},
  booktitle={Numerical methods for the Navier-Stokes equations: Proceedings of the International Workshop Held at Heidelberg, October 25--28, 1993},
  pages={17--30},
  year={1994},
  organization={Springer}
}

@article{chaouqui2022linear,
  title={Linear and nonlinear substructured Restricted Additive Schwarz iterations and preconditioning},
  author={Chaouqui, Faycal and Gander, Martin J and Kumbhar, Pratik M and Vanzan, Tommaso},
  journal={Numerical Algorithms},
  volume={91},
  number={1},
  pages={81--107},
  year={2022},
  publisher={Springer}
}

@inproceedings{lions1988schwarz,
  title={{On the Schwarz alternating method. I}},
  author={Lions, Pierre-Louis and others},
  booktitle={First international symposium on domain decomposition methods for partial differential equations},
  volume={1},
  pages={42},
  year={1988},
  organization={Paris, France}
}

@article{dryja1997nonlinear,
  title={On the nonlinear domain decomposition method},
  author={Dryja, Maksymilian and Hackbusch, Wolfgang},
  journal={BIT Numerical Mathematics},
  volume={37},
  number={2},
  pages={296--311},
  year={1997},
  publisher={Springer}
}

@article{lui1999schwarz,
  title={{On Schwarz alternating methods for nonlinear elliptic PDEs}},
  author={Lui, SH1756041},
  journal={SIAM Journal on Scientific Computing},
  volume={21},
  number={4},
  pages={1506--1523},
  year={1999},
  publisher={SIAM}
}

@article{knoll2004jacobian,
  title={{Jacobian-free Newton--Krylov methods: a survey of approaches and applications}},
  author={Knoll, Dana A and Keyes, David E},
  journal={Journal of Computational Physics},
  volume={193},
  number={2},
  pages={357--397},
  year={2004},
  publisher={Elsevier}
}

@article{maliassov1998schwarz,
xurl = {https://doi.org/10.1515/rnam.1998.13.1.45},
title = {{On the Schwarz alternating method for eigenvalue problems}},
author = {S. Yu. Maliassov},
pages = {45--56},
volume = {13},
number = {1},
journal = {Russian Journal of Numerical Analysis and Mathematical Modelling},
xdoi = {doi:10.1515/rnam.1998.13.1.45},
year = {1998},
lastchecked = {2026-09-07}
}

@inproceedings{luiRecentResultsDomain1996,
  title = {Some Recent Results on Domain Decomposition Methods for Eigenvalue Problems},
  booktitle = {Proc. {{Ninth Int}}. {{Conf}}. on {{Domain Decomposition Methods}}},
  author = {Lui, S. H.},
  year = 1996,
  pages = {426--433}
}

@article{chan2002subspace,
  title={Subspace correction multi-level methods for elliptic eigenvalue problems},
  author={Chan, Tony F and Sharapov, Ilya},
  journal={Numerical linear algebra with applications},
  volume={9},
  number={1},
  pages={1--20},
  year={2002},
  publisher={Wiley Online Library}
}

@article{sleijpen1996jacobi,
  title={{A Jacobi--Davidson iteration method for linear eigenvalue problems}},
  author={Sleijpen, Gerard L. G. and van der Vorst, Henk A.},
  journal={SIAM Journal on Matrix Analysis and Applications},
  volume={17},
  number={2},
  pages={401--425},
  year={1996},
  publisher={SIAM}
}

@article{hwang2010parallel,
  title={{A parallel additive Schwarz preconditioned Jacobi--Davidson algorithm for polynomial eigenvalue problems in quantum dot simulation}},
  author={Hwang, Feng-Nan and Wei, Zih-Hao and Huang, Tsung-Ming and Wang, Weichung},
  journal={Journal of Computational Physics},
  volume={229},
  number={8},
  pages={2932--2947},
  year={2010},
  publisher={Elsevier}
}

@article{zhao2016parallel,
  title={{Parallel two-level domain decomposition based Jacobi--Davidson algorithms for pyramidal quantum dot simulation}},
  author={Zhao, Tao and Hwang, Feng-Nan and Cai, Xiao-Chuan},
  journal={Computer Physics Communications},
  volume={204},
  pages={74--81},
  year={2016},
  publisher={Elsevier}
}

@article{wang2019convergence,
  title={{On the convergence of a two-level preconditioned Jacobi--Davidson method for eigenvalue problems}},
  author={Wang, Wei and Xu, Xuejun},
  journal={Mathematics of Computation},
  volume={88},
  number={319},
  pages={2295--2324},
  year={2019}
}

@article{wang2018two,
  title={A two-level overlapping hybrid domain decomposition method for eigenvalue problems},
  author={Wang, Wei and Xu, Xuejun},
  journal={SIAM Journal on Numerical Analysis},
  volume={56},
  number={1},
  pages={344--368},
  year={2018},
  publisher={SIAM}
}

@article{kalantzis2020domain,
  title={{A domain decomposition Rayleigh--Ritz algorithm for symmetric generalized eigenvalue problems}},
  author={Kalantzis, Vassilis},
  journal={SIAM Journal on Scientific Computing},
  volume={42},
  number={6},
  pages={C410--C435},
  year={2020},
  publisher={SIAM}
}

@article{bach1998quantum,
  title={Quantum electrodynamics of confined nonrelativistic particles},
  author={Bach, Volker and Fr{\"o}hlich, J{\"u}rg and Sigal, Israel Michael},
  journal={Advances in Mathematics},
  volume={137},
  number={2},
  pages={299--395},
  year={1998},
  publisher={Elsevier}
}

@incollection{dusson2021feshbach,
  title={{The Feshbach--Schur map and perturbation theory}},
  author={Dusson, Genevi{\`e}ve and Sigal, Israel Michael and Stamm, Benjamin},
  booktitle={Partial Differential Equations, Spectral Theory, and Mathematical Physics},
  pages={65--88},
  year={2021},
  publisher={European Mathematical Society-EMS-Publishing House GmbH}
}

@article{knyazev2001toward,
  title={Toward the optimal preconditioned eigensolver: Locally optimal block preconditioned conjugate gradient method},
  author={Knyazev, Andrew V},
  journal={SIAM journal on scientific computing},
  volume={23},
  number={2},
  pages={517--541},
  year={2001},
  publisher={SIAM}
}

@book{lang2002introduction,
  title={Introduction to differentiable manifolds},
  author={Lang, Serge},
  year={2002},
  publisher={Springer Science \& Business Media}
}

@article{grigori2026additive,
  title={An additive two-level parallel variant of the DMRG algorithm with coarse-space correction},
  author={Grigori, Laura and Hassan, Muhammad},
  journal={SIAM Journal on Scientific Computing},
  volume={48},
  number={3},
  pages={A1312--A1337},
  year={2026},
  publisher={SIAM}
}

@article{spillane2014abstract,
  title={Abstract robust coarse spaces for systems of PDEs via generalized eigenproblems in the overlaps},
  author={Spillane, Nicole and Dolean, Victorita and Hauret, Patrice and Nataf, Fr{\'e}d{\'e}ric and Pechstein, Clemens and Scheichl, Robert},
  journal={Numerische Mathematik},
  volume={126},
  number={4},
  pages={741--770},
  year={2014},
  publisher={Springer}
}

@article{hou1997multiscale,
  title={A multiscale finite element method for elliptic problems in composite materials and porous media},
  author={Hou, Thomas Y and Wu, Xiao-Hui},
  journal={Journal of computational physics},
  volume={134},
  number={1},
  pages={169--189},
  year={1997},
  publisher={Elsevier}
}

@article{maalqvist2014localization,
  title={Localization of elliptic multiscale problems},
  author={M{\aa}lqvist, Axel and Peterseim, Daniel},
  journal={Mathematics of Computation},
  volume={83},
  number={290},
  pages={2583--2603},
  year={2014}
}

@article{heid2021gradient,
  title={Gradient flow finite element discretizations with energy-based adaptivity for the Gross-Pitaevskii equation},
  author={Heid, Pascal and Stamm, Benjamin and Wihler, Thomas P},
  journal={Journal of computational physics},
  volume={436},
  pages={110165},
  year={2021},
  publisher={Elsevier}
}

@article{badiaGridapExtensibleFinite2020,
  title = {Gridap: {{An}} Extensible {{Finite Element}} Toolbox in {{Julia}}},
  shorttitle = {Gridap},
  author = {Badia, Santiago and Verdugo, Francesc},
  year = {2020},
  month = aug,
  journal = {Journal of Open Source Software},
  volume = {5},
  number = {52},
  pages = {2520},
  issn = {2475-9066},
  xdoi = {10.21105/joss.02520},
  urldate = {2020-09-01},
  langid = {english}
}

@article{theisenScalableTwoLevelDomain2024,
  title = {A {{Scalable Two-Level Domain Decomposition Eigensolver}} for {{Periodic Schr\"odinger Eigenstates}} in {{Anisotropically Expanding Domains}}},
  author = {Theisen, Lambert and Stamm, Benjamin},
  year = 2024,
  month = oct,
  pages = {A3067-A3093},
  publisher = {{Society for Industrial and Applied Mathematics}},
  issn = {1064-8275},
  xdoi = {10.1137/23M161848X},
  urldate = {2024-10-07},
  journal = {SIAM J. Sci. Comput.}
}

@article{anderson1965iterative,
  title={Iterative procedures for nonlinear integral equations},
  author={Anderson, Donald G},
  journal={Journal of the ACM (JACM)},
  volume={12},
  number={4},
  pages={547--560},
  year={1965},
  publisher={ACM New York, NY, USA}
}

@misc{codeZenodo2026,
  author       = {Theisen, Lambert and
                  Benjamin, Stamm and
                  Hassan, Muhammad and
                  Jedrzej, Baraniak},
  title        = {{EnergyMinimizingDD.jl: Energy-minimizing domain
                   decomposition for finite-element problems in Julia
                  }},
  month        = sep,
  year         = 2026,
  publisher    = {Zenodo},
  version      = {v0.1.1},
  xdoi          = {10.5281/zenodo.22942907},
  url          = {https://doi.org/10.5281/zenodo.22942907},
  swhid        = {swh:1:dir:84e0843c86ac13f1c079108c57d43452c3011c77
                   ;origin=https://doi.org/10.5281/zenodo.22937234;vi
                   sit=swh:1:snp:f77c13cbc31c3f354fac3a07b9a195c946ad
                   16c6;anchor=swh:1:rel:0ef5f5500337cdb726c86e373354
                   b0ec1f5d5e08;path=lamBOOO-
                   EnergyMinimizingDD.jl-2f43e48
                  },
}

\end{document}